\documentclass[11pt,a4paper]{article}
\usepackage{amsmath,amssymb,latexsym,amsthm}
\usepackage[english]{babel}
\usepackage[latin2]{inputenc}

\usepackage{amsmath,amsfonts,amssymb}

\begin{document}

\title{$3$-dimensional  Bol loops corresponding to  solvable Lie triple systems}
\author{\'Agota Figula} 
\date{}
\maketitle

\let\cal\mathcal

\newtheorem*{Dfn}{Definition}
\newtheorem{Theo}{Theorem}
\newtheorem{Lemma}[Theo]{Lemma}
\newtheorem{Prop}[Theo]{Proposition}

\theoremstyle{definition}

\newtheorem{rema}{Remark}

\allowdisplaybreaks

\newcommand\bsin{\operatorname{\mathbf{sin}}}
\newcommand\bcos{\operatorname{\mathbf{cos}}}

\newcommand\arc{\operatorname{arc}}
\newcommand\ctg{\operatorname{ctg}}

\begin{abstract} We classify the connected $3$-dimensional differentiable Bol loops $L$ having a solvable Lie group as the
 group topologically generated by the left translations of $L$ using $3$-dimensional solvable Lie
triple systems. Together with \cite{figula} our results complete the
classification of all $3$-dimensional differentiable Bol loops.
\end{abstract}



\noindent
{\footnotesize {{\em Thanks: } This paper has been supported by DAAD.} }

\section{Introduction}

The present research on  differentiable loops is focused  to
such loops  which have local forms determined in a unique way by their  tangential objects.
The most important and most studied class of differentiable loops  are the
Bol loops. Their tangential objects, the Bol algebras, may  be seen as  Lie triple systems
with an additional binary operation (cf.\ \cite{loops} pp.\ 84--86, Def.\ 6.10).
As known the Lie triple systems are in one-to-one correspondence to (global)
simply connected symmetric spaces (cf.\ \cite{symmetric}, \cite{loops} Section~6).
Hence there is a strong connection between the theory of differentiable Bol
loops and the theory of symmetric spaces. In particular the theory of  connected differentiable Bruck loops
(which form a subclass of the class of Bol loops) is essentially the theory of affine
symmetric spaces (cf.\ \cite{loops} Section~11).

The $2$-dimensional differentiable Bol loops are classified in \cite{loops} (Section~25). My goal is to classify
differentiable multiplications satisfying the left Bol identity on  $3$-dimensional connected manifolds
since these manifolds also play an exceptional role.

The $3$-dimensional differentiable Bol loops having a  non-solvable Lie
group as the group topologically generated by the left translations  have
been determined  in \cite{figula}. In this paper I classify all  $3$-dimensional
connected differentiable (global) Bol loops in  which the left translations generate
a solvable Lie group. Since for differentiable Bol loops the group topologically generated by the
left translations is always a Lie group with the results of this paper the
classification of  $3$-dimensional differentiable Bol loops is complete.

We treat the differentiable Bol loops as images of global differentiable sections $\sigma :G/H \to G$, where $G$
is a connected Lie group, $H$ is a closed subgroup containing no non-trivial normal subgroup of $G$ and
 for all $r,s \in \sigma (G/H)$ the element $rsr$ lies in $\sigma (G/H)$.
In this treatment  the exponential images of Lie triple systems form local Bol loops.
Hence for the classification of $3$-dimensional differentiable Bol loops $L$ having a solvable Lie group $G$
as the group topologically generated by the left translations we proceed in the following way: First
we determine all solvable $3$-dimensional Lie triple systems ${\bf m}$ and all enveloping Lie
algebras ${\mathbf g}$ of ${\bf m}$. We show that ${\bf g}$ and  therefore the solvable Lie group $G$ topologically
generated by the left translations of a differentiable Bol loop has dimension four or  five.
Then we find for any pair $({\bf g},{\bf m})$ all subalgebras
 ${\bf h}$ containing no non-trivial ideal of ${\bf g}$ such that
${\bf g}={\bf m} \oplus {\bf h}$ and  we prove  that global Bol loops $L$ correspond
precisely to those exponential images of ${\bf m}$, which form a system of
representatives for the cosets of $\exp {\bf h}$ in $G$.

If the group $G$ is nilpotent then $G$ is the $4$-dimensional non-decompo\-sab\-le nilpotent Lie group and
the corresponding $3$-dimensional Bol loops form only one isotopism class containing precisely two
isomorphism classes (Theorem 4, Section~5.1).

If the solvable Lie group $G$ is $4$-dimensional and not nilpotent then it is a central extension of a
$1$-dimensional Lie group $N$ either by the $3$-dimensional solvable Lie group $G_1$ with precisely two
$1$-dimensional normal subgroups or by the direct product $G_2$ of $\mathbb R$ and the $2$-dimensional
non-abelian Lie group. All loops $L$ corresponding to the extensions of $N$ by $G_1$ are extensions
of $N$ by a loop isotopic to the pseudo-euclidean plane loop (Theorem 6 in Section~5.2 and Theorem 9 in Section~5.3).
The $3$-dimensional Bol loops having the  central extension of $\mathbb R$ by $G_2$ as the group topologically generated
by their left translations are all isomorphic (Theorem 6 in Section~5.2).

If the solvable Lie group $G$ is $5$-dimensional then it is either a semidirect product $G$ of
$\mathbb R^4$ by the group $S=\mathbb R$ such that either no element of $S$ different from the identity
has a real eigenvalue in $\mathbb R^4$ or such that $G$ has a $1$-dimensional centre and precisely
two $1$-dimensional non-central normal subgroups. We prove that for  both groups  $G$  there exist infinitely many
non-isotopic $3$-dimensional differentiable Bol loops corresponding to $G$ (Theorem 7 in Section~5.2 and Theorem 11
in Section~6).

The variety of the $3$-dimensional differentiable Bol loops having a solvable
Lie group as the group topologically generated by their left translations
contains families of loops depending on up to four real parameters.
The size of this variety is so enormous that a classification of
$4$-dimensional differentiable Bol loops having a solvable Lie group as the
group generated  by the left translations seems to be not attainable.

\section{Some basic notions of the theory of Bol loops}

A set $L$ with a binary operation $(x,y) \mapsto x \cdot y$ is called a loop if there exists an element $e \in L$
such that $x=e \cdot x=x \cdot e$ holds for all $x \in L$ and the equations $a \cdot y=b$ and $x \cdot a=b$ have
precisely one solution which we denote by $y=a \backslash b$ and $x=b/a$. The left translation $\lambda _a: y
\mapsto a \cdot y :L \to L$ is a bijection of $L$ for any $a \in L$. Two loops $(L_1, \circ )$ and $(L_2, \ast )$
are called isotopic if there are three bijections $\alpha ,\beta ,\gamma : L_1 \to L_2 $ such that $\alpha (x)
\ast \beta (y)=\gamma (x \circ y)$ holds for any $x,y \in L_1$. Isotopy is an equivalence relation. If $\alpha
=\beta =\gamma $ then the isotopic loops $(L_1, \circ )$ and $(L_2, \ast )$ are called isomorphic. Let $(L_1,
\cdot )$ and $(L_2, \ast )$ be two loops. The set $L=L_1 \times L_2= \{ (a,b) \mid a \in L_1,\ b \in L_2 \}$ with
the componentwise multiplication  is again a loop, which is called the direct product of $L_1$ and $L_2$, and the
loops $(L_1, \cdot )$,  $(L_2, \ast )$ are subloops of $L$.

A loop $L$ is called a Bol loop if for any two left translations
$\lambda _a, \lambda _b$ the product $\lambda _a \lambda _b \lambda _a $ is
again a left translation of $L$. If $L_1$ and $L_2$ are Bol loops, then the
direct product $L_1 \times L_2$ is again a Bol loop.

If the elements of $L$ are points of a differentiable manifold and the operations $(x,y) \mapsto x \cdot y$,
$(x,y) \mapsto x/y$, $ (x,y) \mapsto x \backslash y :L \times L \to L$ are differentiable mappings then $L$ is called a
differentiable loop.

If $L$ is a connected differentiable Bol loop then the group $G$ topologically generated by the left translations
is a connected Lie group (cf.\ \cite{loops}, p.\ 33; \cite{quasigroups}, pp.\ 414--416).

Every connected differentiable Bol loop is isomorphic to a  loop $L$ realized
on the factor space $G/H$, where $G$ is a connected Lie group, $H$ is a
connected closed subgroup containing no normal subgroup $\neq \{1\}$ of $G$
and $\sigma : G/H \to G$ is  a differentiable section with $\sigma (H)=1 \in G$ such that the
subset $\sigma (G/H)$ generates $G$  and for all $r,s \in \sigma (G/H)$ the
element $rsr$ is contained in $\sigma (G/H)$ (cf.\ \cite{loops}, p.~18 and Lemma 1.3, p.\ 17, \cite{kiechle},
Corollary 3.11, p.\ 51). The multiplication  of   $L$ on the factor space $G/H$  is  defined by
$x H \ast y H=\sigma (x H) y H$.

Let  $L_1$ be a loop in the factor space  $G/H$ with respect to the section  $\sigma :G/H \to G$. The loops $L_2$
isomorphic to $L_1$ and having the same set of left translations $\sigma (G/H)$ and the same group $G$ as the group
generated by  $\sigma (G/H)$ correspond to automorphisms $\alpha $ of $G$, which leave $\sigma (G/H)$ invariant. The loop
$L_2$ corresponding to $\alpha $ is realized on $G/ \alpha(H)$ such that the multiplication of $L_2$ is given by
$x  \alpha(H) \ast y \alpha(H)=[\alpha \circ \sigma \circ \alpha _H^{-1}(x \alpha(H))] y \alpha(H)$, where the mapping
$\alpha _H:G/H \to G/ \alpha(H)$ is defined by $kH \to \alpha (k) \alpha(H)$.   Moreover, let $L$ and $L'$
be loops having the same group $G$ generated by their left translations.
Then $L$ and $L'$ are isotopic if and only if there is a loop $L''$
isomorphic to $L'$ having $G$ again as the group generated by its left
translations such that there exists an inner automorphism $\tau $ of $G$
mapping the stabilizer $H''$ of $e'' \in L''$ onto the stabilizer $H$ of
$e \in L$ (cf.\ \cite{loops}, Theorem 1.11, p.\ 21).

\medskip
A real vector space $V$ with a trilinear multiplication $(.\,,.\,,.\,)$ is called a
Lie triple system $\mathcal{V}$, if the following identities are satisfied:
\begin{align}
&(X,X,Y) =0  \\
&(X,Y,Z)+(Y,Z,X)+(Z,X,Y)=0 \\
&( X,Y,(U,V,W) )=( (X,Y,U),V,W ) \nonumber \\
&\hskip 2.85cm +(U,(X,Y,V),W ) +( U,V,(X,Y,W) ). \label{eq:sec2} \end{align}
A Bol algebra $A$ is a Lie triple system $(V,(.\,,.\,,.\,))$ with a bilinear
skew-symmetric operation $[[.\,,.\,]]$, $(X,Y) \mapsto [[X,Y]]:V \times V \to V$
such that the following identity is satisfied:
\begin{multline*}
[[(X,Y,Z),W]]-[[(X,Y,W),Z]]+\big(Z,W,[[X,Y]]\big) \\[1mm]
-\big(X,Y,[[Z,W]]\big)+\big[ \big[ [[X,Y]],[[Z,W]] \big ] \big]=0.  \end{multline*}
With any connected differentiable Bol loop $L$ we can associate a Bol algebra
in the following way: Let $G$ be the Lie group topologically generated by the
 left translations of $L$, and let $(\bf{g, [.\,,.\,]}) $ be the Lie algebra of
$G$. Denote by $\bf{h}$ the Lie algebra of the stabilizer $H$ of the identity $e \in L$ in
$G$ and by $\bf{m} $ $=T_1 \sigma (G/H)$ the tangent space at $1 \in G$ of
the image of the section $\sigma :G/H \to G$ corresponding to the Bol loop  $L$.
Then $\bf{g}=\bf{m} \oplus \bf{h}$, $\big[ [\bf{m},\bf{m}], \bf{m} \big] \subseteq \bf{m}$ and ${\bf m}$
generates the Lie algebra ${\bf g}$. The subspace $\bf{m}$ with the operations
defined by $(X,Y,Z) \mapsto \big[ [X,Y], Z \big]$, $(X,Y) \mapsto [X,Y]_{\bf{m}}$, where $X$, $Y$, $Z$
are elements of  $\bf{m}$ and $Z \mapsto Z_{\bf{m} }: \bf{g} \to \bf{m}$ is the projection of $\bf{g}$ onto $\bf{m}$ along
$\bf{h}$, is the Bol algebra of $L$. The Lie algebra $\bf{g}$ of  $G$ is isomorphic to an enveloping Lie algebra
of the Lie triple system ${\bf m}$ corresponding to  $L$.

An imbedding $T$ of a Lie triple system $\mathcal{V}$ into  a Lie algebra $\mathcal{L}^T$ is a
linear mapping $X \mapsto X^T$ of $\mathcal{V}$ into $\mathcal{L}^T$ such that
\begin{itemize}
\item[(i)]  $(X,Y,Z)^T=[[X^T,Y^T],Z^T]$ holds for all $X,Y,Z \in \mathcal{V}$ and
\item[(ii)]
the image $\mathcal{V}^T$ generates $\mathcal{L}^T$.
\end{itemize}
The Lie algebra  $\mathcal{L}^T$ is called  enveloping Lie algebra of the imbedding $T$. An imbedding $U$ of a Lie triple
system $\mathcal{V}$ is called universal and $\mathcal{L}^U=\mathcal{V}^U \oplus [\mathcal{V}^U,\mathcal{V}^U]$ is a
universal Lie algebra of $\mathcal{V}$ if and only if, for every imbedding $T$ of  $\mathcal{V}$ the mapping
$X^U \mapsto X^T$ is single-valued and can be extended to a Lie algebra homomorphism of $\mathcal{L}^U$ onto
$\mathcal{L}^T$ (\cite{jacobson2}, p.\ 519, and \cite{structure}, p.\ 219).

In \cite{jacobson2} (pp.\ 517--518) it is shown  that for every Lie triple system  $\mathcal{V}$
there exists a  particular imbedding $S$ such that $\sum  _i [X_i^S,Y_i^S]=0$ for
 $X_i, Y_i \in \mathcal{V}$ if and only if $\sum _i(X_i,Y_i,Z)=0$ for every $Z \in \mathcal{V}$. Moreover
 $\mathcal{L}^S=\mathcal{V}^S \oplus [\mathcal{V}^S,\mathcal{V}^S]$. This imbedding is called
the standard imbedding of $\mathcal{V}$ and the Lie algebra  $\mathcal{L}^S$ is the smallest
enveloping algebra. Using the standard imbedding the existence and the uniqueness of a universal imbedding $U$ of
every Lie triple system $\mathcal{V}$ follows (\cite{jacobson2}, p.\ 519). Moreover  if $\mathcal{V}$ is a $n$-dimensional
Lie triple system then the universal Lie algebra $\mathcal{L}^U$ of  $\mathcal{V}$  and therefore every
enveloping Lie algebra  $\mathcal{L}^T$ of  $\mathcal{V}$ has dimension at least $n$ and at most $n(n+1)/2$.

\medskip
A loop $L$ is called a left A-loop if each $\lambda _{x,y}= \lambda_{xy}^{-1}\lambda_x \lambda_y:L \to L$ is an
automorphism of $L$. If $L$ is a differentiable left A-loop then the group $G$
topologically generated by its left translations  is a Lie group
(cf.\ \cite{loops}, Proposition 5.20, p.\ 75). If  ${\bf g}$ is  the Lie algebra of $G$ and  ${\bf h}$ is  the Lie
algebra of the stabilizer $H$ of the identity $e \in L$ in $G$ then one has
$ {\bf m} \oplus {\bf h}= {\bf g}$ and $[{\bf h},{\bf m}] \subseteq {\bf m}$, where ${\bf m}$ is the tangent space $T_e L$
(cf.\ \cite{loops}, Definition 5.18. and Proposition 5.20.  pp.\ 74--75).

A  differentiable loop $L$ is called a Bruck loop if there is an involutory automorphism $\sigma $ of the Lie algebra
${\bf g}$ of the connected Lie group $G$ generated by the left translations of $L$ such that the tangent space
$T_e(L)={\bf m}$ is the $-1$-eigenspace and the Lie algebra ${\bf h}$ of the
stabilizer $H$ of $e \in L$ in $G$ is the $+1$-eigenspace of $\sigma $.

Let $L_1$ be a loop defined on the factor space $G_1/H_1$ with respect to a section
$\sigma _1:G_1/H_1 \to G_1$ the image of which is the set $M_1 \subset G_1$. Let $G_2$ be a  group, let
$\varphi :H_1 \to G_2$ be a homomorphism and $(H_1, \varphi (H_1))=\{(x, \varphi (x)); x \in H_1\}$. A
loop $L$ is called a Scheerer extension of $G_2$ by $L_1$ if $L$ is defined on the factor space
 $(G_1 \times G_2)/(H_1, \varphi (H_1))$ with respect to the section
 $\sigma:(G_1 \times G_2)/(H_1, \varphi (H_1)) \to G_1 \times G_2$ the image of which is the set $M_1 \times G_2$
 (\cite{loops}, Section~2).

\smallskip
From  \cite{figula} we will use often the following fact:

\begin{Lemma}
Let $L$ be a differentiable global loop and denote by ${\bf m}$ the tangent space of
$T_1 \sigma (G/H)$, where $\sigma :G/H \to G$ is the section corresponding to $L$. Then ${\bf m}$ does not contain any
 element of $Ad_{g^{-1}} {\bf h}=g {\bf h} g^{-1}$ for some $g \in G$. Moreover, every element of $G$ can be written
 uniquely as a product of an element of $\sigma (G/H)$ with an element of $H$.
\end{Lemma}

\section{3-dimensional solvable Lie triple systems}

Let $({\bf m}, [[.\,,.\,],.\,])$ be a Lie triple system and let  $({\bf g}^*, [.\,,.\,])$ be
the standard enveloping Lie algebra of $({\bf m}, [[.\,,.\,],.\,])$  (\cite{structure}, p.\ 219).
The isomorphism classes of the $3$-dimensional solvable Lie triple systems and their standard
enveloping Lie algebras may be classified as follows:

\medskip
\noindent
{\bf 1.} If the Lie triple system ${\bf m}$ is abelian then it is the $3$-dimensional abelian Lie
algebra, which is also the standard enveloping Lie algebra of ${\bf m}$ (see Theorem 4.1,
Type I in \cite{bouetou1}).

\medskip
\noindent
{\bf 2.} Since a $3$-dimensional Lie triple system cannot have a $2$-dimensional centre we consider now the case that
${\bf m}$ has a $1$-dimensional centre $\langle e_1 \rangle$. Then the
factor Lie triple system ${\bf m}/ \langle e_1 \rangle$ is $2$-dimensional and according to
\cite{locally} (pp.\ 44--45) it is either abelian or satisfies one of the following relations:
\[
(i) \quad   [[e_2,e_3],e_3]= e_2, \ \qquad  \ (ii) \quad [[e_2,e_3],e_3]=- e_2. \]
It follows that  for ${\bf m}$ and for the corresponding Lie algebra ${\bf g}^*$ we have the following possibilities.

\medskip
\noindent
{\bf 2 a.} If ${\bf m}/ \langle e_1 \rangle$ is abelian then we have $ [[e_2,e_3],e_2]= e_1$, since
${\bf m}$ is not abelian. This Lie triple system is isomorphic to the Lie triple system
belonging to the relation $ [[e_2,e_3],e_3]= e_1$ under the isomorphism $\alpha $ given by
$\alpha (e_1)=e_1$, $\alpha (e_2)=e_3$, $\alpha (e_3)=-e_2$ (see Theorem 4.1, Type II in
\cite{bouetou1}). Then the Lie algebra ${\bf g}^*$ is defined by the following non-trivial relations
\[
[e_2,e_3]=e_4, \quad  [e_4,e_3]=e_1. \]
According to \cite{morozov} (p.\ 162) this is the unique
 $4$-dimensional  nilpotent Lie algebra with $2$-dimensional commutator algebra.

\medskip 
{\bf 2 b.}  The Lie triple system is the direct product of $ \langle e_1 \rangle$ with the $2$-dimensional Lie
triple system satisfying in {\bf 2} either (i) or (ii) respectively.
Using the isomorphism $\alpha $ given by $\alpha (e_1)=e_3$, $\alpha (e_2)=e_1$,
$\alpha (e_3)=e_2$ the Lie triple system with the relation (i) changes into the Lie triple
system ${\bf m}^{+} \times  \langle e_3 \rangle$ satisfying $[[e_1,e_2],e_2]= e_1$ (Type III
in \cite{bouetou1}) and the Lie triple system with the relation (ii) becomes  the Lie triple
system ${\bf m}^{-} \times  \langle e_3 \rangle$ satisfying $[[e_1,e_2],e_2]=- e_1$ (Type III in \cite{bouetou1}).
The Lie algebra ${\bf g}^*_{(+)}$ corresponding to ${\bf m}^{+} \times  \langle e_3 \rangle$ is given by
\[
[e_1,e_2]=e_4, \quad  [e_4,e_2]= e_1, \]
whereas the other products are zero. This  shows that ${\bf g}^*_{(+)}$ is the direct product of the $3$-dimensional
solvable Lie algebra having precisely two $1$-dimensional ideals (\cite{jacobson}, pp.\ 12--14) and
the $1$-dimensional Lie algebra.\\
The Lie algebra ${\bf g}^*_{(-)}$ belonging to ${\bf m}^{-} \times  \langle e_3 \rangle$ is
defined  by
\[
[e_1,e_2]=e_4, \quad   [e_4,e_2]=- e_1, \]
which shows that ${\bf g}^*_{(-)}$ is the direct product of  the $3$-dimensional
solvable Lie algebra having no $1$-dimensional ideal (\cite{jacobson}, pp.\ 12--14) and the
$1$-dimensional Lie algebra.

\medskip
\noindent
{\bf 2 c.}  The Lie triple system is a non-split extension of  $ \langle e_1 \rangle$ by the
$2$-dimensional Lie triple system belonging to the relation (i) or (ii) in {\bf 2} respectively.
Hence it is characterized by
\begin{align*}
{\bf m}^{+} & : [[e_2,e_3],e_2]=e_1, & &  [[e_2,e_3],e_3]= e_2 \quad \qquad \text{or} \\[2mm]
{\bf m}^{-} & :  [[e_2,e_3],e_2]=e_1,  &&  [[e_2,e_3],e_3]=- e_2 \end{align*}
(Type V in \cite{bouetou1}).

The Lie algebra ${\bf g}^*_{(+)}$ of  ${\bf m}^{+}$ is given by
\[
[e_2,e_3]=e_4, \quad  [e_4,e_2]= e_1, \quad   [e_4,e_3]= e_2 \] which shows that  ${\bf g}^{*}_{(+)}$ contains the
$3$-dimensional nilpotent ideal $\langle e_1,e_2, e_4 \rangle$ and the  factor Lie algebra ${\bf g}^*_{(+)}/
\langle e_1 \rangle $ is the $3$-dimensional Lie algebra having precisely two $1$-dimensional  ideals. This Lie
algebra is isomorphic to $g_{4,8}$ with $h=-1$ in \cite{mubarak1} (p.\ 121).

The Lie algebra  ${\bf g}^*_{(-)}$ of  ${\bf m}^{-}$ is defined  by
\[
[e_2,e_3]=e_4, \quad  [e_4,e_2]= e_1, \quad  [e_4,e_3]=-e_2 ,\] which shows that it contains the  $3$-dimensional
nilpotent ideal $\langle e_1,e_2, e_4 \rangle$  and the basis element $e_3$  acts as a euclidean rotation in the
$2$-dimensional subspace $\langle e_2, e_4 \rangle $. This Lie algebra is isomorphic to $g_{4,9}$ with $p=0$ in
\cite{mubarak1} (p.\ 121).

\medskip
\noindent
{\bf 3.} It remains to discuss that ${\bf m}$ has only trivial centre. In this case ${\bf m}$ is determined by
\[
[[e_2,e_3],e_3]=e_1, \quad [[e_3,e_1],e_3]= e_2 \]
(Type VI in  \cite{bouetou1}).

The corresponding Lie algebra ${\bf g}^*$ is defined by:
\[
[e_2,e_3]=e_4,\quad  [e_4,e_3]=e_1, \quad  [e_1,e_3]=e_5, \quad  [e_5, e_3]=-e_2, \]
and the other products are zero. The Lie algebra ${\bf g}^*$ has two  $2$-dimensional  ideals which are invariant
under the action of $e_3$.

\begin{rema}
Our classification of the $3$-dimensional Lie triple system is a slight modification  of
{\sc Bouetou}'s  classification (\cite{bouetou1}). He  has two classes more, namely
\begin{align*}
\text{a)}\qquad \qquad \quad & [[e_2,e_3],e_1]=e_1, & &  [[e_3,e_1],e_2]=-e_1 \qquad\\[3pt]
\text{b)} \qquad\qquad \quad  & [[e_1,e_2],e_2] = \varepsilon e_1, & &  [[e_1,e_2],e_3]= e_1\qquad   \\[3pt]
        \qquad \qquad\quad  & [[e_3,e_1],e_2] = -e_1,& & [[e_3,e_1],e_3] = -\varepsilon e_1,\qquad  \end{align*}
where $\varepsilon =\pm 1$.

The case a) does not satisfy the property (\ref{eq:sec2}) in the definition of a Lie triple system and
the case b) is isomorphic to the case {\bf 2 b} using the isomorphism
\[
\alpha (e_1)=e_1,\quad \alpha(e_2)= \varepsilon e_2 - e_3,\quad \alpha(e_3)=- \varepsilon e_2+(\varepsilon +1) e_3. \]
\end{rema}

\section{3-dimensional Bol loops corresponding to the abelian\\ Lie triple system are abelian groups}

\begin{Lemma}
The universal Lie algebra
${\bf g}^U$ of the abelian Lie triple system ${\bf m}$ is given by the following multiplication table:
\[
[e_1,e_2]=e_4, \quad  [e_1,e_3]=e_5,\quad  [e_2,e_3]=e_6, \]
and the other products are zero.
\end{Lemma}

\begin{proof}
According to the definition of ${\bf g}^U$ we have ${\bf m}^U \cap [{\bf m}^U, {\bf m}^U]=0$. Thus we can choose the
elements $e_1$, $e_2$, $e_3$ as a basis of ${\bf m}^U$ and the elements $e_4:=[e_1,e_2]$, $e_5:=[e_1,e_3]$ and
$e_6:=[e_2,e_3]$ as the generators of $[{\bf m}^U, {\bf m}^U]$. Since ${\bf m}$ is abelian  we obtain the assertion.
\end{proof}

The centre $Z$ of ${\bf g}^U$ is generated by the  elements $e_4$, $e_5$, $e_6$ and is equal to
$[{\bf m}^U,{\bf m}^U]$. Therefore the Lie group $G^U$ of ${\bf g}^U$ is a $6$-dimensional nilpotent Lie group of
nilpotency class $2$. Every enveloping Lie algebra ${\bf g}^T$ of ${\bf m}$ is an epimorphic image of ${\bf g}^U$.
The $4$- or $5$-dimensional epimorphic images of ${\bf g}^U$ are also nilpotent and has nilpotency class $2$.
It follows from  \cite{loops} (p.\ 311) that any global connected differentiable proper Bol loop $L$ having a Lie group of
nilpotency class $2$ as the group topologically generated by its left translations contains an at least $3$-dimensional
nilpotent subgroup. Hence there does not exist any  differentiable proper $3$-dimensional Bol loop $L$ corresponding to
the abelian Lie triple system.

\section{3-dimensional Bol loops belonging to a Lie triple system \\ with  1-dimensional centre}

\subsection{Bol loops corresponding to the non-decomposable nilpotent standard enveloping Lie algebra with dimension 4}

We consider the Lie triple system ${\bf m}$ of type {\bf 2 a}  in Section 3.
\begin{Lemma}
The universal Lie algebra ${\bf g}^U$ of the Lie triple system ${\bf m}$  of type {\bf 2 a} is the
$5$-dimensional nilpotent Lie algebra defined by the following non-trivial products:
\[
[e_2,e_3]=e_4,\quad  [e_4,e_3]=e_1, \quad \ [e_3,e_1]=e_5. \]
The unique $4$-dimensional epimorphic image of ${\bf g}^U$ (up to isomorphisms) is the standard enveloping Lie algebra
${\bf g}^*$ described in {\bf 2 a}.
\end{Lemma}

\begin{proof}
Since ${\bf g}^U={\bf m}^U \oplus [{\bf m}^U, {\bf m}^U]$ we may assume that the set $\{e_1,e_2,e_3 \}$ is the set of the
generators of ${\bf m}^U$ and the elements $e_4:=[e_2,e_3]$, $e_5:=[e_3,e_1]$ and $e_6:=[e_1,e_2]$ are  basis elements
of $[{\bf m}^U, {\bf m}^U]$.  The relations of the Lie triple system of type {\bf 2 a}
yield the following  multiplication table:
\[
[e_2,e_3]=e_4,\quad [e_4,e_3]=e_1, \quad [e_3,e_1]=e_5 , \quad  [e_1,e_2]=e_6. \]
Since $[[e_4,e_3],e_2]+[[e_3,e_2],e_4]+[[e_2,e_4],e_3]=e_6$
this multiplication satisfies the Jacobi identity if and only if $[e_1,e_2]=0$ and this is the
first assertion. The  Lie algebra ${\bf g}^U$ is nilpotent hence every epimorphic images of ${\bf g}^U$ is also
nilpotent. If ${\bf g}$ is a $4$-dimensional epimorphic image of  ${\bf g}^U$ then the commutator
subalgebra of ${\bf g}$ is image of the commutator subalgebra $({\bf g}^U)'$.  Since
$ \dim ({\bf g}^U)'=3$ we have $ \dim {\bf g}'=2$ and ${\bf g}$ is  the standard enveloping Lie algebra
${\bf g}^*$ (cf.\ {\bf 2 a}).
\end{proof}

Denote by $G$ the Lie group of the standard enveloping Lie algebra ${\bf g}^*$. Using the Campbell--Hausdorff series the
multiplication of  $G$ is defined by:
\begin{gather*}
(x_1,x_2,x_3,x_4) \ast (y_1,y_2,y_3,y_4 ) \\[2mm]
=\left ( \begin{array}{c}
x_1+y_1+\dfrac{1}{2}(x_4 y_3-x_3 y_4)+\dfrac{1}{12}( x_3^2 y_2-x_3 x_2 y_3)+ \dfrac{1}{12}(x_2 y_3^2-x_3 y_3 y_2) \\[1mm]
x_2+y_2 \\[1mm]
x_3+y_3 \\[1mm]
x_4+y_4+\dfrac{1}{2}(x_2 y_3-x_3 y_2) \end{array} \right ) \end{gather*}
(\cite{bouetou2}, p.\ 77).
A $1$-dimensional subalgebra ${\bf h}$ of ${\bf g}^*$ such that ${\bf h}$ does not contain any non-trivial ideal of
${\bf g}$ and  ${\bf h} \cap {\bf m}= \{ 0 \}$ holds has the form
\[
{\bf h}=\langle e_4+a_1 e_1+a_2 e_2 +a_3 e_3 \rangle, \quad a_i \in \mathbb R. \]
The automorphism group  of ${\bf g}$ consisting of the linear mappings
\[
\alpha (e_1)=b f^2 e_1,\quad  \alpha (e_2)=a e_1+b e_2,\quad \alpha (e_3)=d e_1+l e_2+f e_3,
 \quad   \alpha (e_4)=b f e_4, \]
where $a,b,d,l,f \in \mathbb R$ and $bf \neq 0$, leaves the subspace ${\bf m}=\langle e_1, e_2, e_3 \rangle $
invariant and maps the subalgebra ${\bf h}$ onto one of the following subalgebras
\[
{\bf h}_1=\langle e_4 \rangle ,\quad  {\bf h}_2=\langle e_4+e_1 \rangle ,\quad
{\bf h}_3=\langle e_4+e_2 \rangle ,\quad   {\bf h}_4=\langle e_4+e_3 \rangle  \]
(see \cite{bouetou2}). Since the  element $e_4+e_2 \in {\bf h}_3$ is conjugate to the element
$e_2- \frac{1}{2} e_1 \in {\bf m}$ under $g=( 0,0,-1,0) \in G$ and the element  $e_4+e_3 \in {\bf h}_4$ is conjugate to
the element $ e_3 \in {\bf m}$ under $g=(0,1,0,0) \in G$ we have a contradiction to  Lemma 1.
Therefore we have to consider only the cases $({\bf g}^*, {\bf h}_1)$ and $({\bf g}^*, {\bf h}_2)$.
In  \cite{bouetou2} it is proved that for these $2$ cases global Bol loops exist. The loop $L$ belonging to the triple
\[
(G,H_1= \exp {\bf h}_1= \{ (0, 0, 0, h ) \mid h \in \mathbb R \},\
\exp {\bf m}=\{ (a,b,c,0) \mid a,b,c \in \mathbb R\}) \]
is a Bruck loop. The  loop $L^*$ corresponding to
\[
(G,H_2= \exp {\bf h}_2=\{(h,0,0,h) \mid h \in \mathbb R \},\ \exp {\bf m}=\{(a,b,c,0) \mid a,b,c \in \mathbb R\}) \]
is a left A-loop, because of $[{\bf h}, {\bf m}] \subseteq {\bf m}$. But it is not a Bruck loop since there is no
involutory automorphism $\sigma :{\bf g} \to {\bf g}$ such that $\sigma ({\bf m})=- {\bf m}$ and
$\sigma ({\bf h}_2)={\bf h}_2$.

Since the conjugation by the element $g=( 0,0,-1,0) \in G$ maps the subalgebra ${\bf h}_1$ of $H_1$
onto the subalgebra ${\bf h}_2$ of $H_2$ the loop $L$ is isotopic to $L^*$.

\medskip
Now we consider the universal Lie algebra ${\bf g}^U$ defined in Lemma 3, which is  the Lie algebra $L_5^2$ in
\cite{morozov} (p.\ 162). Using the Campbell--Hausdorff series (\cite{reinsch}) the multiplication of the Lie group $G^U$
of  ${\bf g}^U$ is given as follows:
\begin{gather*}
(x_1,x_2,x_3,x_4,x_5) \ast  (y_1,y_2,y_3,y_4,y_5) \\[2mm]
=\left ( \begin{array}{c}
x_1+y_1+\dfrac{1}{2}(x_4 y_3-x_3 y_4)+\dfrac{1}{12}( x_3^2 y_2-x_3 x_2 y_3)+ \dfrac{1}{12}(x_2 y_3^2-x_3 y_3 y_2) \\[1mm]
x_2+y_2 \\[1mm]
x_3+y_3 \\[1mm]
x_4+y_4+\dfrac{1}{2}(x_2 y_3-x_3 y_2) \\[1mm]
x_5+y_5+\dfrac{1}{2}(x_3 y_1-x_1 y_3)+ \dfrac{1}{12}(- x_3^2 y_4+x_3 x_4 y_3) \\[1mm]
+\dfrac{1}{12}(-x_4 y_3^2+x_3 y_3 y_4)+ \dfrac{1}{24}(x_2 x_3 y_3^2-x_3^2 y_2 y_3)  \end{array} \right ). \end{gather*}
The class of the $2$-dimensional subalgebras ${\bf h}$ of ${\bf g}_1$, which does not
contain any non-trivial ideal and ${\bf h} \cap {\bf m}= \{ 0 \}$ has the following shape:
\[
{\bf h}_{a,b,a',b'}=\langle e_4+ a e_1 +b e_2, e_5 +a' e_1 +b' e_2 \rangle,\quad  a,b,a',b' \in \mathbb R,
 \ (a', b') \neq (0,0) \]
(\cite{bouetou2}, p.\ 80). There is no Bol loop $L$ such that the group topologically generated
by the left translations of $L$ is the group $G^U$ and the stabilizer of the identity  $e \in L$ in $G^U$ is  the group
\[
H_{a,b,a',b'}=\{(\lambda_1a+\lambda_2 a',\lambda_1 b+\lambda_2 b',0,\lambda_1,\lambda_2),\lambda_1,\lambda_2\in\mathbb R\},
 \quad a, b, a', b' \in \mathbb R , \]
where $(a',b') \neq (0,0)$. Namely we show  that for given  $a, b, a', b' \in \mathbb R $ with  $(a', b')\neq (0,0)$ we can
find $(0,0) \neq (\lambda _1, \lambda _2)  \in \mathbb R^2$  and an element $x=(x_1,x_2,x_3,x_4,x_5) \in G^U$ such that
\[
Ad_x ( \lambda_1( e_4+ a e_1 +b e_2) + \lambda_2 (e_5 +a' e_1 +b' e_2)) \in {\bf m} \backslash \{ 0 \} \]
where ${\bf m}=\{ y_1 e_1+y_2 e_2+y_3 e_3;\  y_1,y_2,y_3 \in \mathbb R \}$.
This is a consequence of  the fact  that  the following system of equations:
\begin{gather*}
y_1=\lambda_1 \left(a -\frac{1}{2} x_3\right) + \lambda _2 a', \quad y_2= \lambda _1 b + \lambda _2 b', \quad y_3=0 \\[2mm]
\lambda_1 (1-x_3 b)- \lambda_2 b' x_3=0,\quad \lambda_2 (1+x_3 a') + \lambda_1 \left(x_3 a - \frac{1}{3} x_3^2\right)
 =0 \end{gather*}
has a solution  $x_3 \neq 0$, $(\lambda_1, \lambda_2) \neq (0,0)$ and $(y_1,y_2,y_3) \neq (0,0,0)$
which holds true  since  there exists  $x_3 \neq 0$ such that
\[
1+ x_3 (a' -b) + x_3^2 ( b' a - b a') - \frac{1}{3} b' x_3^3=0. \]
Summarizing our discussion we obtain

\begin{Theo}
There is only one isotopism class $\cal {C}$ of  $3$-dimensional connected differentiable Bol loops such that the group
$G$ topologically generated by their left translations is a nilpotent Lie group. The group $G$ is isomorphic to the
$4$-dimensional non-decomposable nilpotent Lie group. The class  $\cal {C}$ consists of precisely two isomorphism classes
${\cal C}_1$ and ${\cal C}_2$ which may be represented by the Bruck loop $L$ having the group
$H= \{(0,0,0,h) \mid h \in \mathbb R \}$ as the stabilizer of $e \in L$ in $G$ respectively  by the left A-loop $L^*$
having the group $H=\{ (h, 0, 0, h ) \mid h \in \mathbb R \}$ as the stabilizer
of $e \in L^*$ in $G$.
\end{Theo}

\subsection{Bol loops corresponding to a Lie triple system which is a direct product of its centre and a non-abelian
Lie triple system}

We  discuss here the Lie triple systems  characterized in {\bf 2 b} in Section 3.

\begin{Lemma}
The universal Lie algebras ${\bf g}^U_{(+)}$ and ${\bf g}^U_{(-)}$  of the Lie triple systems
${\bf m}^+\times\langle e_3 \rangle$ or  ${\bf m}^- \times \langle e_3 \rangle$ respectively, are defined by:
\[
[e_1,e_2]=e_4, \quad [e_4,e_2]= \varepsilon\  e_1, \quad   [e_2,e_3]=e_5, \]
where $\varepsilon =1$ for ${\bf g}^U_{(+)}$ and $-1$ for ${\bf g}^U_{(-)}$, and the other products are zero.

The unique $4$-dimensional epimorphic image of  ${\bf g}^U_{(-)}$ is
(up to isomorphisms) the standard enveloping Lie algebra ${\bf g}^*_{(-)}$ described in {\bf 2 b}.

The $4$-dimensional epimorphic images of  ${\bf g}^U_{(+)}$ are (up to isomorphisms)
either the standard enveloping Lie algebra ${\bf g}^*_{(+)}$ given in  {\bf 2 b}  or the
Lie algebra ${\bf g}$ given by:
\[
[e_1,e_2]=e_1, \quad  [e_2,e_3]=e_4, \]
whereas  the other products are zero.
\end{Lemma}

\begin{proof}
For a  basis  of the universal Lie algebras ${\bf g}^U={\bf m}^U \oplus [{\bf m}^U,{\bf m}^U]$ one can choose the elements
$e_1$, $e_2$, $e_3$, $e_4$, $e_5$, $e_6$, where $e_1, e_2, e_3$ are the generators of ${\bf m}^U$ and
$e_4:=[e_1,e_2]$, $e_5:=[e_2,e_3]$, $e_6:=[e_1,e_3]$ are the generators of
$[{\bf m}^U,{\bf m}^U]$. Using the relations of the Lie triple systems of type {\bf 2 b}  we obtain the following
multiplication table:
\[
[e_1,e_2]=e_4, \quad [e_4,e_2]= \pm e_1, \quad  [e_2,e_3]=e_5, \quad  [e_1,e_3]=e_6 \]
and the other products are zero. Since for the elements $e_2,e_3,e_4$ one has
\[
[[e_2,e_3],e_4]+[[e_3,e_4],e_2]+[[e_4,e_2],e_3]= \pm e_6, \]
this multiplication satisfies the Jacobi identity precisely if $[e_1,e_3]=0$,  and we obtain
the universal Lie algebras ${\bf g}^U_{(\pm)}$. The unique $1$-dimensional ideal of  ${\bf g}^U_{(-)}$ is the centre  of
${\bf g}^U_{(-)}$, which is generated by $e_5$. Moreover, the epimorphic image
$\alpha ({\bf g}^U_{(-)})$ under the mapping $\alpha (e_i)=e_i$, $i=1,2,3,4$,
 $\alpha (e_5)=0$ is the Lie algebra ${\bf g}^*_{(-)}$.

The  $1$-dimensional ideals of ${\bf g}^U_{(+)}$ are $i_1=\langle e_5 \rangle$,
$i_2=\langle e_1+e_4 \rangle$, $i_3=\langle e_4-e_1 \rangle$. The image of  ${\bf g}^U_{(+)}$ under the epimorphism $\beta
(e_i)=e_i$, $i=1,2,3,4$ and $\beta (e_5)=0$ is the Lie algebra ${\bf g}^*_{(+)}$.  The Lie algebras
${\bf g}^U_{(+)}/ \langle e_1+e_4 \rangle$ and ${\bf g}^U_{(+)}/ \langle e_4-e_1 \rangle$ are determined by
\begin{align*}
[e_1,e_2]& =-e_1, & &  [e_2,e_3]=e_4; \qquad \text{and  by} \\[2mm]
[e_1,e_2] & =e_1, &&   [e_2,e_3]=e_4 \end{align*}
respectively. This shows that ${\bf g}^U_{(+)}/ \langle e_1+e_4 \rangle$ is isomorphic to
${\bf g}^U_{(+)}/ \langle e_4-e_1 \rangle$ under the isomorphism $\gamma (e_i)=e_i$, $i=1,4$ and
$\gamma (e_j)=-e_j$, $j=2,3,$ and the assertion follows.
\end{proof}

First we seek for Bol loops having the standard enveloping Lie algebra ${\bf g}^*_{(+)}$ given
in {\bf 2 b} as the Lie algebra of the group topologically generated by their left translations.
The Lie group $G$ of ${\bf g}^*_{(+)}$ is the direct product $G=G_1 \times G_2$, where $G_1$ is the $3$-dimensional
solvable Lie group having precisely two $1$-dimensional normal subgroups and $G_2$ is a  $1$-dimensional Lie group.
Since the Lie triple system is the direct product of its centre $C$ and a $2$-dimensional non-abelian Lie triple system
$A$ one has $\exp {\bf m}= \exp {\bf m}_1 \times \exp {\bf m}_2$, where $\exp {\bf m}_1$ respectively  $\exp {\bf m}_2$
corresponds to $A$ respectively to $C$. Moreover, $\exp {\bf m}_1 \subseteq G_1$ and $\exp {\bf m}_2=G_2$.

First we assume that the $1$-dimensional Lie group $H=\exp{\bf h}$ is contained in $G_1 \times \{ 1 \}$. Then the loop
$L$ is the direct product of a $2$-dimensional Bol loop $L_1$ and a $1$-dimensional Lie group (\cite{loops},
Proposition 1.19, p.\ 28). The loop $L_1$ has $G_1$ as the group generated by its left translations, and it is isomorphic
to precisely one of the non-isomorphic loops $L_{\alpha }$, $\alpha \in \mathbb R$ with $\alpha \le -1$ given in
 Theorem 23.1 of  \cite{loops}. All loops $L_{\alpha }$ and hence also $L_1$ are isotopic to the pseudo-euclidean plane
loop  (\cite{loops}, Remark 25.4, p.\ 326).

If the $1$-dimensional Lie group $H=\exp {\bf h}$ is not contained in $G_1 \times \{1\}$ then $H$ is isomorphic to
$\mathbb R$ since $G_1$ does not contain any discrete normal subgroup $\neq 1$. Therefore  $G_2 \cong  \mathbb R$,
$\exp {\bf m}=\exp {\bf m}_1  \times \mathbb R$  and $H$ has the shape $\{ (h_1, \varphi (h_1)\mid h_1 \in H_1 \}$, where
$H_1 \cong \mathbb R$ is a subgroup of $G_1$ and $\varphi :H_1 \to G_2$ is a monomorphism. For a loop
$L$ corresponding to the pair $(G,H)$ the group $G_2$ is a normal  subgroup
of  $L$ and the factor loop $L / G_2$ is isomorphic to a loop $L_1$ defined on the
factor space $G_1/H_1$. According to  Theorem 23.1 in \cite{loops} the loop $L_1$ is isomorphic to a loop $L_{\alpha }$.
 Then the Proposition  2.4 in \cite{loops} yields that  $L$ is a  Scheerer extension of the
group $\mathbb R$ by a loop  $L_{\alpha }$.

\medskip
Now we deal with  the standard enveloping Lie algebra ${\bf g}^*_{(-)}$ given in {\bf 2~b}. The Lie  group $G$ of
${\bf g}^*_{(-)}$ is the direct product $G=G_1 \times G_2$ of the $3$-dimensional solvable Lie group $G_1$ having no
non-trivial normal subgroup and a $1$-dimensional Lie group $G_2$. Since $\exp {\bf m}$ decomposes into the
topological product $\exp {\bf m}= \exp {\bf m}_1 \times \exp { \bf m}_2$ with  $\exp {\bf m}_1 \subset G_1$ and
$\exp { \bf m}_2=G_2$ the $1$-dimensional Lie group $H$ has the form $(H_1, \varphi (H_1))$, where $\varphi :H_1 \to G_2$
is a homomorphism. Hence  the loop belonging to $(G,H, \exp {\bf m})$ is  a Scheerer extension of a
$1$-dimensional Lie group and a $2$-dimensional loop $\tilde{L}$ (cf.\ \cite{loops} Proposition~1.19, p.\ 28 and
Proposition~2.4, p.\ 44). But the group $G_1$ cannot be the group topologically generated by the left
translations  of  $\tilde{L}$ (cf.\ \cite{loops} Lemma 23.15, p.\ 312).
Therefore there is no differentiable Bol loop corresponding to the group $G$.

\medskip
Now we investigate the Lie algebra ${\bf g}$ in Lemma 5, which
 consists of the matrices
$$
v e_1+u e_2+z e_3+k e_4 \mapsto \left ( \begin{array}{ccccc}
0 & v & 0 & 0 & 0 \\
0 & u & 0 & 0 & 0 \\
0 & 0 & 0 & u & k \\
0 & 0 & 0 & 0 & z \\
0 & 0 & 0 & 0 & 0 \end{array} \right );\quad u,v,k,z \in \mathbb R. $$
It is a  central extension of $\mathbb R$ by the direct product of $\mathbb R$ and the non-abelian $2$-dimensional
Lie algebra (see  \cite{mubarak1}, pp.\ 120--121).
The multiplication of the Lie group $G$ of ${\bf g}$ is defined by
\[
(x_1,x_2,x_3,x_4) \ast (y_1,y_2,y_3,y_4)=(y_1+x_1 e^{y_2}, x_2+y_2, x_3+y_3, x_4+y_4+x_2 y_3). \]
The $1$-dimensional subalgebras ${\bf h}$ of ${\bf g}$ which  complement ${\bf m}=\langle e_1,e_2,e_3 \rangle$ have the
shapes:
\[
{\bf h}_{a_1,a_2,a_3}=\langle e_4+a_1 e_1+a_2 e_2+a_3 e_3 \rangle, \]
where $a_1,a_2,a_3 \in \mathbb R$. For $a_1=a_2=a_3=0$ the Lie algebra ${\bf h}_{0,0,0}=\langle e_4 \rangle $
 is an ideal of ${\bf g}$. Therefore we have $(a_1,a_2,a_3) \neq (0,0,0)$. The automorphisms $\gamma $ of ${\bf g}$
 leaving ${\bf m}$ invariant are determined by the linear mappings
$$
\gamma (e_1)=a e_1,\quad \gamma (e_2)=b_1 e_1+e_2+b_3 e_3, \quad \gamma (e_3)=d e_3,\quad \gamma (e_4)=d e_4,$$
such that $a,d \in \mathbb R \backslash \{ 0 \}$ and $b_1,b_3 \in \mathbb R$. A suitable automorphism
$\gamma $ of ${\bf g}$ with $\gamma ({\bf m})={\bf m}$ maps the subalgebra ${\bf h}_{a_1,a_2,a_3}$
onto one of the following Lie algebras:
\[
{\bf h}_1=\langle e_4+e_2 \rangle ,\ {\bf h}_2=\langle e_4+a_3 e_3 \rangle ,\ a_3 \in \mathbb R  \backslash \{0\},
 \quad {\bf h}_3=\langle e_4+e_1+a_3 e_3 \rangle ,\ a_3 \in \mathbb R .\]
Because of  $e_2=Ad_g(e_4+e_2) \in {\bf m}$ with $g=(0,0,-1,0) \in G$ the Lie algebra ${\bf h}_1$ is excluded. Since for
$a_3 \neq 0$ and $g=(0,a_3^{-1},0,0) \in G$ one has $a_3 e_3=Ad_g(e_4+a_3 e_3) \in {\bf m}$ and
$[\exp (a_3^{-1})] e_1+a_3 e_3 = Ad_g(e_4+e_1+a_3 e_3) \in {\bf m}$ we have to investigate only the triple
$({\bf g}, {\bf h}=\langle e_4+e_1 \rangle, {\bf m})$ (cf.\ Lemma 1). For the exponential image of
${\bf m}=\langle e_1,e_2,e_3 \rangle$ we obtain
\begin{align*}
\exp\ {\bf m}& =\exp \{k_1 e_1+k_2 e_2+k_3 e_3; k_1,k_2,k_3 \in \mathbb R \} \\
&=\left \{ \left (k_1 \frac{e^{k_2}-1}{k_2}, k_2, k_3, \frac{1}{2} k_2 k_3 \right ), \ k_i \in \mathbb R,\ i=1,2,3\right \},
 \end{align*}
and the subgroup $H=\exp  \{a(e_4+e_1), a \in \mathbb R \}$ consists of the  elements $(a,0,0,a)$, $a \in \mathbb R$.

Since any element of $G$ decomposes uniquely as $(0,y_1,y_2,y_3)(a,0,0,a)$ we can conclude that $\exp {\bf m}$ determines a
 global  Bol loop if and only if  each element $g=(0,y_1,y_2,y_3) \in G$, $y_i \in \mathbb R$, $i=1,2,3$  can  be
written  uniquely as a product $g=m h$   or equivalently $m=g h^{-1}$  with $m \in \exp {\bf m}$ and
$h \in H$. This is the case since  for all given $y_1,y_2,y_3 \in \mathbb R$
the following system of equations
\[
y_1=k_2,\ y_2=k_3,\ y_3-a=\frac{1}{2} k_2 k_3,\ a=-k_1 \frac{e^{k_2}-1}{k_2} \]
has a unique solution $(a,k_1,k_2,k_3) \in \mathbb R^4$ given by
\[
k_2:=y_1,\ k_3:=y_2,\ a:=y_3-\frac{1}{2} y_1 y_2,\ k_1:=\frac{\frac{1}{2}y_1 y_2- y_3}{\frac{e^{y_1}-1}{y_1}}. \]
Hence the pair $(G,H=\{(a,0,0,a), a \in \mathbb R \} )$ corresponds to a
 $3$-dimensional Bol loop $L$. Because of $[{\bf h}, {\bf m}] \subseteq {\bf m}$ the loop $L$ is a left A-loop.

\smallskip
Now we summarize the discussion in

\begin{Theo}
Let $L$ be a $3$-dimensional connected differentiable Bol loop corresponding to a Lie triple system which is a
direct product of its centre and a non-abelian $2$-dimensional Lie triple system. If the group $G$ topologically
generated by the left translations of $L$ is $4$-dimensional,  then for $L$ and for  $G$ precisely one of the
following cases occur:

{\rm 1)} $G$ is the direct product of the $3$-dimensional solvable Lie group having precisely two $1$-dimensional
normal subgroups and a $1$-dimensional Lie group and $L$ is  either the direct product of the $1$-dimensional
compact Lie group $SO_2(\mathbb R)$ with a $2$-dimensional Bol loop $L_{\alpha }$ defined in  Theorem 23.1 of
\cite{loops}  or a Scheerer extension of the  group $\mathbb R$ by a  loop $L_{\alpha }$.

{\rm 2)} $G$ is the $4$-dimensional solvable Lie group with  the multiplication
\[
(x_1,x_2,x_3,x_4) \ast (y_1,y_2,y_3,y_4)=(y_1+x_1 e^{y_2}, x_2+y_2, x_3+y_3, x_4+y_4+x_2 y_3) \] and $L$ is
isomorphic to  the left A-loop  having  $H=\{ (a,0,0,a)\mid a \in \mathbb R \}$ as the stabilizer of the identity
of $L$.
\end{Theo}

Finally we treat the universal Lie algebras ${\bf g}^U_{(\pm )}$ defined in Lemma 5. (The Lie algebra ${\bf g}^U_{(+)}$
is isomorphic to the Lie algebra $g_{5,8}$ with $\gamma =-1$ and  ${\bf g}^U_{(-)}$ is isomorphic to
the Lie algebra $g_{5,14}$ with $p=0$ in \cite{mubarak}, p.\ 105.)
 The multiplication of the Lie group $G^U_{(\pm)}$ corresponding to ${\bf g}^U_{(\pm)}$   is given by:
\[\left ( \begin{array}{c}
x_1 \\
x_2 \\
x_3 \\
x_4 \\
x_5  \end{array} \right ) \ast \left ( \begin{array}{c}
y_1 \\
y_2 \\
y_3 \\
y_4 \\
y_5 \end{array} \right )= \left ( \begin{array}{c}
y_1+ x_1 \bcos y_2 + \varepsilon x_4 \bsin  y_2  \\
y_2+ x_2 \\
y_3+ x_3 \\
y_4+ x_1 \bsin y_2 + x_4 \bcos  y_2  \\
y_5+x_5+ x_2 y_3 \end{array} \right ). \]
The triple $(\bcos  y_2, \bsin  y_2, \varepsilon)$ denotes $(\cosh y_2, \sinh y_2, 1)$ in case  $G^U_{(+)}$ and \\
$(\cos y_2, \sin y_2, -1)$ in case $G^U_{(-)}$.

The $2$-dimensional subalgebras ${\bf h}$ of ${\bf g}^U_{(\pm)}$   which are  complements to
${\bf m}=\langle e_1,e_2,e_3 \rangle$ have the shapes:
\[
{\bf h}_{a_1,a_3,b_1,b_3}=\langle e_4 + a_1 e_1 + a_3 e_3, e_5 + b_1 e_1 + b_3 e_3 \rangle, \]
where $a_1,a_3,b_1,b_3 \in \mathbb R$. Since the ideal $\langle e_5 \rangle$ of ${\bf g}^U_{(\pm)}$ lies in
${\bf h}_{a_1,a_3,0,0}$ and the ideal  $\langle e_4 \pm e_1 \rangle$ of
${\bf g}^U_{(+)}$ is contained in  ${\bf h}_{\pm 1,0,b_1,b_3}$
 we may suppose that $(b_1,b_3) \neq (0,0)$ in the case of
${\bf g}^U_{(+)}$ as well as of  ${\bf g}^U_{(-)}$ and
$(a_1,a_3) \neq (\pm 1,0)$ in the case ${\bf g}^U_{(+)}$.

For $b_1=0$  the element $0 \neq b_3 e_3 \in {\bf m}$  is conjugate to  $e_5+b_3 e_3 \in {\bf h}$ under
$g=(0,- b_3^{-1},0,0,0) \in G^U_{(\pm)}$ which  contradicts Lemma~1.

If  $b_1 \neq 0$ then the linear mapping $\alpha $ defined by
\[
\alpha (e_1)= \frac{1}{b_1} e_1,\ \alpha (e_2)= e_2,\ \alpha (e_3)= e_3,\ \alpha (e_4)= \frac{1}{b_1} e_4,\
\alpha (e_5)= e_5 \]
is an automorphism of  ${\bf g}^U_{(\pm)}$. This automorphism  leaves the subspace ${\bf m}$ invariant and
reduces  ${\bf h}_{a_1,a_3,b_1,b_3}$ to ${\bf h}_{a_1,a_3,1,b_3}$.

The Lie group $H_{a_1,a_3,1,b_3}=\exp {\bf h}_{a_1,a_3,1,b_3}$ consists of the elements
\[
\{(l a_1+k,0,l a_3+k b_3, l,k),\ l,k \in \mathbb R \} \]
and the exponential image of the subspace ${\bf m}$ has the form
\begin{align*}
\exp\ {\bf m}& =\exp \{k_1 e_1+k_2 e_2+k_3 e_3; k_1,k_2,k_3 \in \mathbb R \} \\[2mm]
& =\left\{\left ( \frac{k_1 \bsin  k_2}{k_2}, k_2, k_3, \varepsilon \frac{k_1(\bcos  k_2-1)}{k_2},
 \frac{1}{2}k_2 k_3\right),\ k_1,k_2,k_3 \in \mathbb R \right\}. \end{align*}
Every element of the Lie group $G^U_{(\pm)}$  can be written  uniquely as a  product
\[
(x_1,x_2,x_3,x_4,x_5)=(0,f_2,f_3,0,f_5)(l a_1+k,0,l a_3+k b_3, l,k), \]
where $(l a_1+k,0,l a_3+k b_3, l,k) \in H_{a_1,a_3,1,b_3}$. Each  element $g=(0,f_2,f_3,0,f_5)$, $f_i \in \mathbb R$ for
$i=2,3,5,$  has in $G^U_{(\pm)}$ a unique decomposition as   $g=m\ h$ or equivalently $m=g\ h^{-1}$ with
$m \in \exp {\bf m}$, $h \in H_{a_1,a_3,1,b_3}$ if and only if for all given
$f_2,f_3,f_5,a_1,a_3,b_3 \in \mathbb R$ the following system of equations
\begin{gather}
-l a_1-k=\frac{k_1 \bsin f_2}{f_2},\quad k_3=f_3-l a_3-k b_3,
 \quad l=-\varepsilon \frac{k_1 (\bcos f_2-1)}{f_2},\nonumber \\
-k+f_5+f_2(k_3-f_3)=\tfrac{1}{2} f_2 k_3, k_2=f_2 \tag{$*$} \end{gather}
has a unique solution  $(k_1,k_2,k_3,k,l) \in \mathbb R^5$.

In the group $G^U_{(-)}$ we find
\begin{gather}
k_2=f_2,\quad  k_1=\tfrac{f_2 (-2 f_5+ f_2 f_3)}{\tilde{n}}, \nonumber \\
k_3=\tfrac{2[(\cos f_2 -1)(f_3 a_1-f_3 a_3 f_2+f_3b_3f_2a_1+a_3 f_5-b_3 a_1 f_5)+\sin f_2 (f_3+f_3 b_3 f_2-b_3 f_5)]}
 {\tilde{n}}, \nonumber \\
k=\tfrac{ (2 f_5- f_2 f_3)[\sin f_2+a_1(\cos f_2 -1)]}{\tilde{n}}, \quad
 l=\tfrac{(\cos f_2 -1) (-2 f_5+ f_2 f_3)}{\tilde{n}}, \tag{1} \end{gather}
where $\tilde{n}=(\cos f_2 -1)(2 a_1-a_3 f_2+b_3 f_2 a_1)+(2+b_3 f_2)\sin f_2 $.

In  $G^U_{(+)}$  the system $(\ast )$ has the following solution:
\begin{gather}
k_2=f_2,\quad  k_1=\tfrac{2 f_2 e^{f_2} (-2 f_5+ f_2 f_3)}{(e^{f_2}-1)n}, \nonumber \\
k_3=\tfrac{2[(e^{f_2} -1)(-f_3 a_1+f_3 a_3 f_2-f_3b_3f_2a_1-a_3 f_5+b_3 a_1 f_5)
 +(e^{f_2}+1) (f_3+f_3 b_3 f_2-b_3 f_5)]}{n}, \nonumber \\
k=\tfrac{ (-2 f_5+ f_2 f_3)(a_1 e^{f_2}-e^{f_2}-a_1-1)}{n}, \quad l=\tfrac{(e^{f_2}-1)(2f_5-f_2 f_3)}{n},\tag{2}\end{gather}
where $n=(1-e^{f_2})(2 a_1-a_3 f_2+b_3 f_2 a_1)+(e^{f_2}+1) (2+b_3 f_2)$.

The solution (1) respectively the solution (2) is unique  if and only if $\tilde{n} \neq 0$ respectively $n \neq 0$. If
for a value $f_2$ one has $n(f_2)=0$ respectively $n'(f_2)=0$ then the coset $(0,f_2,f_3,0,f_5)  H_{a_1,a_3,1,b_3}$
contains no element of $\exp {\bf m}$.

Considering $f_2$ as a variable $x$ for the function $\tilde{n}(f_2)=\tilde{n}(x)$ one has
$\tilde{n}(x)=0$ if and only if  $a_3(x)=\big( \frac{2}{x}+b_3\big) \big( a_1+\frac{\sin x}{\cos x -1} \big)$,
where  $a_1,b_3 \in \mathbb R$ and $x \in \mathbb R \backslash \{2 \pi l \}$, $l \in \mathbb Z$. For all $a_1 \in\mathbb R$
the function $h(x):= a_1+\frac{\sin x}{\cos x -1}$  has period $2 \pi$.  It  is continuous and  strictly
increasing on the intervals $(2 \pi l, 2 \pi +2 \pi l )$, $l \in \mathbb Z$ such that
 $\lim _{x \searrow 2 \pi l} h(x)=-\infty$ and $\lim _{x \nearrow 2 \pi+2 \pi l} h(x)=\infty$.
The function $ \frac{2}{x}+b_3 $ is  for $b_3 \le -\frac{2}{3 \pi}$ continuous and negative in  $(4 \pi,6 \pi)$ and for
$b_3 > -\frac{2}{3 \pi}$ it is continuous and positive in  $(0,2 \pi)$. Hence  the restriction of the function $a_3(x)$ to
$(4 \pi, 6 \pi)$ respectively to  $(0,2 \pi)$  takes all real numbers as values. This means that for all given $a_1,a_3,b_3$
there is a value $p  \in \mathbb R \backslash \{2 \pi l \}$, $l \in \mathbb Z$ such that $\tilde{n}(p)=0$.

Replacing $f_2$ by the variable $x$ we investigate the function $n(f_2)=n(x)$. We have $n(0)=4$. We seek for
$p \in \mathbb R \backslash \{ 0 \}$ with  $n(p)=0$. Since $n(x)$ is continuous it is enough to prove that there is
$x \in \mathbb R \backslash \{ 0 \}$ with $n(x)<0$. This happens  for the following triples
$$
\begin{array}{llll}
\hbox{a)} & (b_3=0,\ a_3=0,\ a_1 \notin [ -1,1])\quad & \hbox{b)} & (b_3=0, a_3<0,\ a_1 \in \mathbb R)\\[2mm]
\hbox{c)} & (b_3 \in \mathbb R \backslash \{ 0 \},\ a_3 \le 0,\ a_1 \in \mathbb R) &
\hbox{d)} &  \Big(b_3<0,\ a_3>0,\ a_1 < \dfrac{a_3}{b_3}+1\Big) \\[2mm]
\hbox{e)} & \Big(b_3>0,\ a_3>0,\ a_1 > \dfrac{a_3}{b_3}-1\Big). & & \end{array}$$
Namely, in the case a)  $\lim _{x \to -\infty} \frac{n(x)}{e^x+1}<0$ for $a_1<-1$ and
$\lim _{x \to \infty} \frac{n(x)}{e^x+1}<0$ for $a_1>1$. In the cases b) and e) we have
$\lim _{x \to -\infty} \frac{n(x)}{e^x+1}=-\infty$ and in the case d) one obtains
$\lim _{x \to \infty} \frac{n(x)}{e^x+1}=-\infty$. Moreover, in the case c) one has $n(-\frac{2}{b_3}) \le 0$.
Thus for the above triples $(a_1,a_3,b_3)$ there is $p \in \mathbb R \backslash \{ 0 \}$ such that $n(p)=0$.

Let $\sigma :G^U_{(\pm )}/H_{a_1,a_3,1,b_3} \to G^U_{(\pm )}$ be a section belonging to a differentiable Bol loop $L$
with dimension $3$. If $\sigma (G^U_{(\pm)}/H_{a_1,a_3,1,b_3})$ contains $\exp{\bf m}$ then any coset
$(0,f_2,0,0,1) H_{a_1,a_3,1,b_3}$, $(f_2 \in \mathbb R)$ should contain precisely one element $s$ of
$\sigma (G^U_{(\pm)}/H_{a_1,a_3,1,b_3})$. For $f_2 \neq p$ we obtain in the case  $G^U_{(-)}$
\[
s=\left( -2 \frac{\sin f_2}{\tilde{n}}, f_2, k_3, \frac{2(\cos  f_2-1)}{\tilde{n}}, \frac{1}{2} f_2 k_3 \right) \]
and in the case  $G^U_{(+)}$
\[
s=\left( \frac{-2 (e^{ f_2}+1)}{n}, f_2, k_3, \frac{-4 e^{f_2}(\cosh  f_2-1)}{(e^{f_2}-1) n},\frac{1}{2} f_2 k_3 \right). \]
Since  $\sigma$ is continuous one has
\[
\sigma ((0,p,0,0,1) H_{a_1,a_3,1,b_3})=\lim _{f_2 \to p} \sigma ((0,f_2,0,0,1) H_{a_1,a_3,1,b_3})= \lim _{f_2 \to p} s. \]
But  $\lim _{f_2 \to p}  \frac{2(\cos  f_2-1)}{\tilde{n}}= \infty$ as well as
$ \lim _{f_2 \to p} \frac{-2 (e^{ f_2}+1)}{n} = \infty$ which are contradictions. Therefore the group $G^U_{(-)}$ cannot
be the group topologically generated by the  left translations of a differentiable $3$-dimensional Bol loop and for the
group  $G^U_{(+)}$ the parameters satisfying the conditions a) till e) are excluded.

\medskip
Now for  $G^U_{(+)}$ it remains to investigate the triples
$$
\begin{array}{llll}
\hbox{(i)} & (b_3=0,\ a_3=0,\ -1<a_1 <1) \quad& \hbox{(ii)} & (b_3=0,\ a_3>0,\ a_1 \in \mathbb R) \\[2mm]
\hbox{(iii)}& \Big(b_3<0,\ a_3>0,\ a_1 > \dfrac{a_3}{b_3}+1\Big) &
 \hbox{(iv)} & \Big(b_3>0,\ a_3>0,\ a_1 < \dfrac{a_3}{b_3}-1\Big)\\[2mm]
\hbox{(v)} & \Big(b_3>0,\ a_3>0,\ a_1 = \dfrac{a_3}{b_3}-1\Big) &
\hbox{(vi)} &  \Big(b_3<0,\ a_3>0,a_1 = \frac{a_3}{b_3}+1\Big).  \end{array}$$

In the  case (i)  the function $n(x)$ is positive. Therefore there is a  connected differentiable $3$-dimensional Bol loop,
which is realized on the factor space $G^U_{(+)}/H_{a_1,0,1,0}$  with  $-1<a_1 <1$.

\medskip
In  the case (ii)  we have
\[
n(x)=e^x(x a_3-2 a_1+2)-x a_3+2 a_1+2 \]
and for the derivations we obtain
\begin{align*}
n'(x)& = e^x(x a_3-2 a_1+2+a_3)-a_3 \\
n''(x)& = e^x(x a_3-2 a_1+2+2a_3) \\
n'''(x)& = e^x(x a_3-2 a_1+2+3a_3). \end{align*}
Since  $n''(x)=0$  only for  $u=\frac{2a_1-2-2a_3}{a_3}$ holds  and $n'''(u)=a_3>0$ the function $n'(x)$
assumes in $u$ its unique minimum.  Moreover, we have
\begin{align*}
& \lim _{x \to \infty} n'(x)=\infty, & & \lim _{x \to - \infty} n'(x)<0, \qquad \text{and} \\
& \lim _{x \to \infty} n(x)=\infty, & &  \lim _{x \to - \infty} n(x)=\infty. \end{align*} Therefore there is only
one value $p$ for which $n'(p)=0$ and in $p$ the function  $n(x)$ achieves its  unique minimum. One obtains
$n'(p)=0$ if and only if  $a_1=\frac{1}{2}(p a_3+2+a_3-\frac{a_3}{e^p})$. Furthermore,
 we have  $n(p)>0$ if and only if  $p=0$ or $0<a_3<\frac{4 e^p}{(e^p-1)^2}$ if  $p \in \mathbb R \backslash \{ 0 \}$.
Thus for the parameters $(a_3,a_1)$ satisfying the properties
$$
0<a_3  \quad  \hbox{and} \quad  a_1=1 $$
or
$$ 0<a_3<\frac{4 e^p}{(e^p-1)^2} \ \hbox{and} \ a_1=\frac{1}{2} \left (p a_3+2+a_3-\frac{a_3}{e^p} \right ) $$
there is a  connected differentiable $3$-dimensional Bol loop corresponding to the pair $(G^U_{(+)},H_{a_1,a_3,1,0})$.

\medskip
In the cases (iii) and (iv) we have
\[
n(x)=(e^{x}+1) (b_3 x+2)+ (1-e^{x})(x b_3 a_1-x a_3+2 a_1) \]
and for the derivations one obtains
\begin{align*}
n'(x)&=e^x(x(b_3+a_3-b_3 a_1)+b_3+a_3-b_3 a_1+2-2 a_1)+b_3+b_3 a_1 -a_3 \\
n''(x)&=e^x(x(b_3+a_3-b_3 a_1)+2b_3+2a_3-2b_3 a_1+2-2 a_1) \\
n'''(x)&=e^x(x(b_3+a_3-b_3 a_1)+3b_3+3a_3-3b_3 a_1+2-2 a_1). \end{align*}
The same arguments as above show that the function $n'(x)$ has only one minimum
in  $\frac{2(b_3 a_1-b_3-a_3+a_1-1)}{b_3+a_3-b_3 a_1}$ and that there exists only
one value $p$ such that  $n'(p)=0$; for this value $p$ the function  $n(x)$ takes its  unique minimum.

We have $n'(p)=0$ if and only if
\[
p=0 \quad \hbox{and} \quad  a_1=b_3+1 \]
or
\[
a_3=\frac{e^p(1+p)(b_3 a_1-b_3)+e^p(2 a_1-2)-b_3-b_3 a_1}{e^p(1+p)-1} \quad \hbox{if}\  p \in \mathbb R \backslash\{0 \}. \]
Putting $a_3$ into the expression of $n(x)$ we obtain the following: For the value $p$ one has  $n(p)>0$
if and only if one of the following  cases is satisfied
\begin{flalign*}
& \text{(I)} \quad && p=0 \ \ \text{and} \ \ a_1=b_3+1 \\
& \text{(II)} \quad && e^p(1+p)-1 <0 \ \text{and}\ p^2 e^p b_3 \,{-}\, a_1(e^{2p} +1)+ e^{2p}+2 p e^p -1 +2 a_1 e^p <0 \\
& \text{(III)} \quad && e^p(1+p)-1 >0 \ \ \text{and} \\
&     &&p^2 e^p b_3 - a_1(e^{2p}+1)+ e^{2p}+2 p e^p -1 +2 a_1 e^p >0.
\end{flalign*}
In the  case  (I) the conditions in (iii)  reduce to\\[1mm]
(iii)\ a) \quad $b_3<0, \ a_1=b_3+1, \ b_3^2<a_3$ \\[1mm]
and from the conditions in (iv) one gets\\[1mm]
(iv)\  b) \quad $b_3>0, \ a_1=b_3+1, \ b_3^2+2 b_3<a_3$. \\[1mm]
In both cases there is a connected differentiable $3$-dimensional Bol loop $L$ realized on the factor space
 $G^U_{(+)}/H_{a_1,a_3,1,b_3}$.

\medskip
Now we discuss the case (II). For the parameters satisfying (iii)  it is equivalent to the following system of
inequalities
\begin{flalign*}
&(\alpha ) \qquad p<0,\ \ b_3<0,\ \ a_1 b_3 < a_3+b_3, \quad  (\beta ) \quad a_3>0,&& \\
&(\gamma ) \qquad b_3<\frac{ a_1(e^{p}-1)^2- e^{2p}-2p e^p +1}{p^2 e^p}, &&\\
&(\delta ) \qquad a_3=\frac{e^p(1+p)(b_3 a_1-b_3)+e^p(2 a_1-2)-b_3-b_3 a_1}{e^p(1+p)-1}.&& \end{flalign*}
Using   $(\delta )$ the condition $(\alpha )$ may replaced by
\begin{flalign*}
(\alpha ') \quad a_1 <1, \quad  e^p(a_1 -1)<b_3 <0, \quad p<0. &&\end{flalign*}
The condition $(\beta )$
is satisfied if and only if
\begin{flalign*}
(\beta ') \quad \varepsilon b_3 <  \varepsilon \frac{e^p(2-2 a_1)}{e^p(1+p)(a_1-1)-(1+a_1)}\quad \text{and}\quad
\varepsilon a_1< \varepsilon \frac{1+e^p(1+p)}{-1+e^p(1+p)} &&\end{flalign*}
with $\varepsilon \in \{ 1,-1 \}$ holds. Since  $p<0$  the condition  $a_1 <1$ gives in  $ (\beta ')$ for $\varepsilon =1$
that $a_1< \frac{1+e^p(1+p)}{-1+e^p(1+p)}$ and for $\varepsilon =-1$ that $ \frac{1+e^p(1+p)}{-1+e^p(1+p)}< a_1 <1$.
Therefore the expression $ \frac{e^p(2-2 a_1)}{e^p(1+p)(a_1-1)-(1+a_1)}$ is positive for $\varepsilon =1$ and
negative for $\varepsilon =-1$.

Let $f(p)$, $l(p,a_1)$ and $k(p,a_1)$ be the following functions
\begin{align*}
f(p)&:= \frac{1+e^p(1+p)}{-1+e^p(1+p)}, \quad l(p,a_1):=e^p( a_1 -1),\\
 k(p,a_1)& := \frac{e^p(2-2 a_1)}{e^p(1+p)(a_1-1)-(1+a_1)}. \end{align*}
Thus for $\varepsilon =1$ the conditions $(\alpha ')$ and  $ (\beta ')$  yield
\begin{itemize}
\item[a)] \quad $l(p,a_1)<b_3<0$ \ \  and \ \ $a_1< f(p)$ \end{itemize}
whereas for $\varepsilon =-1$ the conditions $(\alpha ')$ and  $ (\beta ')$  give
\begin{itemize}
\item[b)] \quad $f(p)<a_1<1$ \ \  and \ \ $k(p,a_1)<b_3<0$\ \ which\  satisfy \ \ $ (\gamma )$. \end{itemize}
The function $n(p,a_1):=\frac{ a_1(e^{p}-1)^2- e^{2p}-2p e^p +1}{p^2 e^p}$ in $(\gamma )$ is non negative if and only if
\begin{itemize}
\item[(A)] \quad $a_1 \ge \dfrac{e^{2p}+2p e^p -1}{(e^p -1)^2}$. \end{itemize}
Denote by $g(p)$ the function $g(p)=\frac{e^{2p}+2p e^p -1}{(e^p -1)^2}$. Using for  all $p<0$ the inequality
\begin{itemize}
\item[(B)] \quad $g(p) < f(p) $ \end{itemize}
one sees that the condition a) holds if and only if  one of the following  systems of inequalities  is satisfied:
\begin{itemize}
\item[c)] \quad $g(p) \le a_1 <f(p)$\ \  and \ \ $l(p,a_1)< b_3<0$,
\item[d)] \quad $a_1 <g(p)$\ \  and \ \  $l(p,a_1)< b_3<n(p,a_1)$ \ \ \hbox{if} \ \ $l(p,a_1)< n(p,a_1)$. \end{itemize}
Because of $ p^2 e^{2p}-(e^p-1)^2<0$ for all $p<0$,  the condition  $l(p,a_1)< n(p,a_1)$ is satisfied if and only if
\[
\frac{ p^2 e^{2p} - e^{2p}-2 p e^p +1}{ p^2 e^{2p}- e^{2p}+2  e^p -1} < a_1. \]
Let $h(p)$ be the function $h(p)=\frac{ p^2 e^{2p} - e^{2p}-2 p e^p +1}{ p^2 e^{2p}- e^{2p}+2  e^p -1}$. Since $h(p)<g(p)$
for  all  $p<0$  the condition d) is satisfied if and only if
\begin{itemize}
\item[e)] \quad $h(p)<a_1 <g(p)$ \ \  and \ \ $l(p,a_1) < b_3 <n(p,a_1)$ \ \ holds. \end{itemize}
Thus for $p<0$ and $b_3<0$ there is a connected differentiable Bol loop $L$ such that the group topologically generated
by its left translations is the group $G^U_{(+)}$ and the stabilizer of $e \in L$ is the subgroup $H_{a_1,a_3,1,b_3}$
if and only if the parameters $a_1,a_3, b_3$ satisfy one of the systems of inequalities  b), c) or e) and
the condition~$(\delta )$.

\smallskip
For the parameters (iv) the case (II) yields the following system of inequalities
\begin{align*}
&(\alpha )\quad   p<0,\ b_3>0,\ a_1 b_3 < a_3-b_3,\qquad   (\beta) \quad a_3>0, \\
&(\gamma ) \quad b_3< n(p,a_1), \qquad
 (\delta) \quad a_3=\frac{e^p(1\,{+}\,p)(b_3 a_1{-}\,b_3)\,{+}\,e^p(2 a_1\,{-}\,2)\,{-}\,b_3\,{-}\,b_3 a_1}{e^p(1+p)-1}.
  \end{align*}
Using   $(\delta )$ the condition $ (\alpha )$ holds if and only if one of
the following cases is satisfied
\begin{flalign*}
&(\alpha ') \quad p<-1,\ \  a_1<1, \ \  0<b_3<\frac{a_1-1}{1+p},&&\\
& (\alpha '') \quad p=-1, \ \  a_1<1,\ \  0<b_3, && \\
& (\alpha ''') \quad -1<p<0,\ \ \max\left\{0,\frac{a_1-1}{1+p}\right \} <b_3. && \end{flalign*}
The condition $(\beta )$ may be replaced by\\[1mm]
$(\beta ') \quad \varepsilon b_3 <\varepsilon  k(p,a_1)\ \ \hbox{and} \ \ \varepsilon a_1< \varepsilon  f(p)$\\[1mm]
with $\varepsilon \in \{ 1,-1 \}$. Denote by $m(p,a_1)$ the function $\frac{a_1-1}{1+p}$. The conditions $ (\alpha ')$ and
$(\beta')$, $(\alpha'')$ and $(\beta')$, $(\alpha''')$ and $(\beta')$ yield for $\varepsilon=1$ the corresponding
conditions
\begin{itemize}
\item[a)] \quad $p<-1$,\ \ $a_1<f(p)$,\ \ $0< b_3< \min\{k(p,a_1),m(p,a_1) \}$,
\item[b)] \quad $p=-1$,\ \ $a_1<-1$,\ \ $0< b_3<k(-1,a_1)$,
\item[c)] \quad $-1<p<0$,\ \ $a_1<f(p)$,\ \ $0< b_3< k(p,a_1)$
\end{itemize}
and  for $\varepsilon =-1$ the  conditions
\begin{itemize}
\item[d)] \quad $p<-1$,\ \ $f(p)<a_1<1$, \ \ $0< b_3< k(p,a_1)$,
\item[e)] \quad $p=-1$,\ \ $a_1<1$,\ \ $0< b_3$,
\item[f)] \quad $-1<p<0$,\ \ $1<a_1$,\ \ $\max\{m(p,a_1),k(p,a_1) \} <b_3$,
\item[g)] \quad $-1<p<0$,\ \ $f(p)<a_1 \le 1$, \ \ $0< b_3$.
\end{itemize}
Now we deal with the condition $(\gamma )$. Using the inequalities $(A)$ and
$(B)$ the conditions a) till g) hold if and only if  the following
conditions in the same order as a) till g) are satisfied:
\begin{itemize}
\item[a')] \quad $p<-1$,\ \ $g(p) \le a_1<f(p)$,\ \ $0< b_3< \min\{k(p,a_1),m(p,a_1),n(p,a_1) \}$,
\item[b')] \quad $p=-1$,\ \ $g(-1) \le a_1<-1$,\ \ $0< b_3<\min \{k(-1,a_1),n(-1,a_1) \}$,
\item[c')] \quad $-1<p<0$,\ \ $g(p) \le a_1<f(p)$,\ \ $0< b_3<\min \{k(p,a_1),n(p,a_1) \}$,
\item[d')] \quad $p<-1$,\ \ $f(p)<a_1<1$, \ \ $0< b_3<\min \{ k(p,a_1),n(p,a_1) \}$,
\item[e')] \quad $p=-1$,\ \ $g(-1) \le a_1<1$,\ \ $0< b_3<n(-1,a_1)$,
\item[f')] \quad $-1<p<0$,\ \ $1<a_1$,\ \ and \ \ $\max\{m(p,a_1),k(p,a_1) \} <b_3 < n(p,a_1)$,
\item[ ]\quad  if \ \ $\max\{m(p,a_1),k(p,a_1) \} <b_3 < n(p,a_1)$,
\item[g')] \quad $-1<p<0$,\ \ $f(p)<a_1 \le 1$, \ \ $0< b_3< n(p,a_1)$. \end{itemize}
Since for $-1<p<0$ and $1<a_1$ one has $k(p,a_1)< m(p,a_1)$ as well as\newline  $(1+p)(e^p-1)^2-p^2 e^p<0$ the
inequality $\max \{m(p,a_1),k(p,a_1) \} <b_3 < n(p,a_1)$ in f') is satisfied if and only if
\[
a_1<\frac{(1+p)(e^{2p}+2 p e^p-1)-p^2 e^p}{(1+p)(e^{2p}-2 e^p+1)-p^2 e^p}. \]
The function $v(p)=\frac{(1+p)(e^{2p}+2 p e^p-1)-p^2 e^p}{(1+p)(e^{2p}-2 e^p+1)-p^2 e^p} $ is  greater than $1$ for
$-1<p<0$. Hence the condition f') is equivalent to\\[1mm]
h') $\quad -1<p<0,\ \ 1<a_1<v(p),\ \  m(p,a_1)<b_3 <n(p,a_1)$. It follows that  for $p<0$ and $b_3>0$ there is a
differentiable Bol loop $L$ defined on the factor space $G^U_{(+)}/H_{a_1,a_3,1,b_3}$ if and only if the
parameters $a_1$, $a_3$, $b_3$ satisfy one of the systems of inequalities  a') till h')  and the condition
$(\delta )$.

\smallskip
Now we discuss the case (III). For (iii)  we obtain  the following  system of inequalities
\begin{flalign*}
&(\alpha ) \qquad p>0,\ \ b_3<0,\ \ a_1 b_3 < a_3+b_3, &&\\
&(\beta ) \qquad a_3>0, \quad \quad (\gamma ) \quad \quad b_3> n(p,a_1), && \\
&(\delta ) \qquad a_3=\frac{e^p(1+p)(b_3 a_1-b_3)+e^p(2 a_1-2)-b_3-b_3 a_1}{e^p(1+p)-1}. && \end{flalign*}
Using $ (\delta )$ the condition $ (\alpha )$ yields\\[1mm]
$(\alpha ')$  \quad $b_3 <\  \hbox{min} \{0, e^p(a_1-1) \}$ \ \ and \ \ $p>0$.\\[1mm]
Furthermore,  $(\beta )$ is satisfied  if and only if\\[1mm]
$(\beta ')$ \quad $\varepsilon b_3 >\varepsilon  k(p,a_1)$
\ \hbox{and} \  $ \varepsilon a_1>  \varepsilon f(p) $ \\[1mm]
with $\varepsilon\in\{1,-1\}$ holds. Since $p>0$ the conditions $(\alpha')$ and $(\beta')$ give for $\varepsilon =1$\\[1mm]
a) \quad $a_1>f(p)$ \ and \  $k(p,a_1)<b_3<0$ \\[1mm]
whereas   for $ \varepsilon =-1$ we obtain  one of the following conditions\\[1mm]
b) \quad $1<a_1<f(p)$\  and \  $b_3<0$\\[1mm]
c) \quad $a_1<1$\ and \ $b_3< \min\{l(p,a_1),k(p,a_1) \}$.\\[1mm]
Since for $a_1<1$ and $p>0$ we have $l(p,a_1) <  k(p,a_1)$  the condition c) yields\\[1mm]
d) \quad $ a_1 <1$ \ \hbox{and} \ $b_3 <l(p,a_1)$.\\[1mm]
Now we investigate the condition  $(\gamma )$. The function $n(p,a_1)$ is non negative if and only if\\[1mm]
(C) \quad $a_1 \ge \frac{e^{2p}+2p e^p -1}{(e^p -1)^2}.$ \\[1mm]
Because of\\[1mm]
(D) \quad $ f(p)< g(p)$ \  for all \ $p>0$  \\[1mm]
the condition a) may be replaced by\\[1mm]
e) \quad $f(p)<a_1<g(p)$ \ and \ $\max  \{k(p,a_1),n(p,a_1) \}<b_3<0$. \\[1mm]
Moreover the condition b) is equivalent to\\[1mm]
f) \quad $1<a_1<f(p)$\  and \  $n(p,a_1)<b_3<0$ \\[1mm]
whereas the condition d) is equivalent to\\[1mm]
g) \quad $ a_1 <1$ \  and \ $n(p,a_1)<b_3 <l(p,a_1)$ \ for \ $n(p,a_1)<l(p,a_1)$. \\[1mm]
Since for  $p>0$ one has
\[
p^2 e^{2p}-(e^p-1)^2>0  \quad \hbox{and} \quad \   h(p) < 1 \]
the relation $n(p,a_1)<l(p,a_1)$ holds if and only if  $h(p)<a_1 $. Using  this  inequality and $h(p)<1$ the condition
g) is equivalent to\\[1mm]
h) \quad $h(p)<a_1<1$ \ and \ $n(p,a_1)<b_3<l(p,a_1)$. \\[1mm]
Thus for  $p>0$ and $b_3<0$ there  exists a  differentiable  Bol loop, which is realized on the factor space
$G^U_{(+)}/H_{a_1,a_3,1,b_3}$ if and only if $a_1$, $a_3$, $b_3$  satisfy one of the systems of inequalities  e),
f) or h) and the condition $(\delta )$.

\smallskip
For the parameters (iv)  the case (III) is equivalent to the following system of inequalities
\begin{flalign*}
&(\alpha ) \qquad p>0,\ \ b_3>0,\ \ a_1 b_3 < a_3-b_3,&& \\
&(\beta ) \qquad a_3>0, \qquad \qquad (\gamma ) \quad \quad b_3>n(p,a_1), & &\\
&(\delta ) \qquad a_3=\frac{e^p(1+p)(b_3 a_1-b_3)+e^p(2 a_1-2)-b_3-b_3 a_1}{e^p(1+p)-1}. && \end{flalign*}
Using $ (\delta )$ the condition $ (\alpha )$ may be replaced by the condition\\[1mm]
$(\alpha ')$  \quad $1<a_1$, \ $0<b_3<m(p,a_1)$ \ and \ $p>0$. \\[1mm]
Furthermore,  the condition $(\beta )$ is satisfied  if and only if\\[1mm]
$(\beta ')$ \quad $ \varepsilon b_3 >\varepsilon  k(p,a_1)$
\ \hbox{and} \  $ \varepsilon a_1>  \varepsilon f(p)$ \\[1mm]
with $\varepsilon\in\{1,-1\}$ holds. Since $p>0$ the conditions $(\alpha')$ and $(\beta')$ give for $\varepsilon=1$\\[1mm]
a)\quad $f(p)<a_1$ \ and \ $0<b_3< m(p,a_1)$ \\[1mm]
and  for $ \varepsilon =-1$\\[1mm]
b) \quad $1 < a_1 < f(p)$\ and \ $0<b_3<k(p,a_1)$.\\[1mm]
Now we deal with the property $(\gamma )$. Using the inequalities (C) and (D) one sees that  the inequalities in b)
satisfy  $(\gamma )$ and that the condition a) holds if and only if one of the following cases is true:\\[1mm]
c) \quad $f(p)< a_1 \le g(p)$ \ and \ $0 < b_3< m(p,a_1)$\\[1mm]
d) \quad $g(p)<a_1 $ \ and \ $n(p,a_1) < b_3 < m(p,a_1)$ \ if \ $n(p,a_1) < m(p,a_1)$. \\[1mm]
Since  $(1+p)(e^p-1)^2-p^2 e^p>0$ for $p>0$ the condition $n(p,a_1) < m(p,a_1)$ is equivalent to $a_1<v(p)$. Moreover, for
$p>0$ one has $g(p) <v(p)$ and the condition d) is satisfied if and only if\\[1mm]
e) \quad $g(p)<a_1<v(p)$ \ and \  $n(p,a_1)< b_3 < m(p,a_1)$. \\[1mm]
Hence for $p>0$ and $b_3>0$  there exists a differentiable  Bol loop $L$ having  $G^U_{(+)}$ as the group
topologically generated by the left translations and the subgroup $H_{a_1,a_3,1,b_3}$ as the stabilizer of
$e \in L$ in $G^U_{(+)}$ if and only if the parameters $a_1$, $a_3$, $b_3$ satisfy one of the
conditions  b), c) or e) and $(\delta )$.

\smallskip
For the parameters (v) we have $n'(p)=0$ if and only if
\[
p=0 \ \text{and} \  \frac{a_3}{b_3}=b_3+2 \ \hbox{or} \ a_3=b_3(p b_3+b_3+2)\quad \hbox{if} \
 p \in \mathbb R \backslash \{ 0 \}.  \]
Hence  $n(p)>0$ if and only if one of the following cases holds true:\\[1mm]
1) \quad $b_3>0$, $a_3=b_3(b_3+2)$, $a_1=\frac{a_3}{b_3} -1$ \ if \  $p=0$\\[1mm]
and\\[1mm]
2) \quad $b_3(p+1-e^p)+2>0$ \ for $p \in \mathbb R \backslash \{ 0 \}$.

\smallskip
For the parameters  in 1) there is a  differentiable  Bol loop $L$ having  $G^U_{(+)}$ as the group topologically
generated by its  left translations and the group $H_{a_1,a_3,1,b_3}$ as the stabilizer  in $G^U_{(+)}$.

\medskip
The case 2) is equivalent to the following system of inequalities\\[1mm]
$(\alpha )$ \ \ $b_3>0$, \ \ $b_3(p+1-e^p)+2>0$, \ \ $(\beta )$ \ \ $a_3>0$, \ \  $a_3=b_3(p b_3+b_3+2)$. \\[1mm]
Because of $p+1-e^p<0$ for all $p \in \mathbb R \backslash \{ 0 \}$ the condition $(\alpha )$ may be replaced by\\[1mm]
$(\alpha ')$ \quad $0<b_3<- \frac{2}{p+1-e^p}$. \\[1mm]
The condition $(\beta )$ is satisfied if and only if one of the following holds:
\begin{flalign*}
&(\beta') \qquad p>-1\ \text{and}\ b_3>-\frac{2}{p+1}, \quad\qquad
 (\beta '') \qquad p=-1 \ \text{and}\ b_3>0, && \\
& (\beta ''') \qquad p<-1\ \text{and}\ b_3<-\frac{2}{p+1}. && \end{flalign*}
Comparing the conditions $(\alpha')$ and $(\beta')$ respectively $(\alpha')$ and $(\beta'')$ we obtain that for
$p \ge -1$ one has $0< b_3< -\frac{2}{p+1-e^p}$. Since  $-\frac{2}{p+1}> -\frac{2}{p+1-e^p}$  for all $p<-1$
holds $(\alpha ')$ and $(\beta ''')$ reduces to $0< b_3< -\frac{2}{p+1-e^p}$. Hence for  $p\in\mathbb R \backslash\{0 \}$
there exists a  differentiable  Bol loop  realized on the factor space $G^U_{(+)}/H_{a_1,a_3,1,b_3}$ if and only if
\[
0< b_3< -\frac{2}{p+1-e^p},\quad a_3=b_3(p b_3+b_3+2),\quad a_1=\frac{a_3}{b_3} -1. \]

For the parameters (vi) we have $n''(x)=-2 e^x a_3 b_3^{-1}>0 $  for all $x \in \mathbb R$. Hence the function
$n'(x)=-2 e^x a_3 b_3^{-1}+2 b_3$ is strongly monotone increasing. Thus $n'(x)=0$ is  satisfied only for  $p=\ln
(b_3^2 a_3^{-1})$   and  $n(p)>0$ if and only if
\[
b_3 \left(\ln \frac{b_3^2}{a_3}-1 \right) +2+ \frac{a_3}{b_3}>0. \]
This condition is necessary and sufficient that a group $H_{a_1,a_3,1,b_3}$
with parameters in (vi) is the stabilizer of a differentiable Bol loop
realized on the factor spaces  $G^U_{(+)}/H_{a_1,a_3,1,b_3}$.

\smallskip
From the above discussion we obtain the main part of the following

\begin{Theo}
Let $L$ be a $3$-dimensional connected differentiable Bol loop corresponding to a solvable Lie triple system  which
is the  direct product of its centre and a non-abelian $2$-dimensional Lie triple system. If the group $G$
topologically generated by the left translations of $L$ is at  least $5$-dimensional then
 $G$ is the $5$-dimensional solvable Lie group  defined  by:
\begin{gather*}
(x_1, x_2, x_3, x_4,x_5) \ast (y_1,y_2,y_3,y_4,y_5)
=(y_1+ x_1 \cosh  y_2 + x_4 \sinh  y_2,\\ y_2+ x_2,y_3+ x_3,
y_4+ x_1 \sinh  y_2 + x_4 \cosh  y_2, y_5+x_5+ x_2 y_3 ). \end{gather*}
Let \\[1mm]
{\rm(a)}  \quad  $H_{a,0,0}=\{ (la+k,0,0,l,k);\ l,k \in \mathbb R \}$,\ $-1 < a<1$, \\[1mm]
{\rm (b)} \quad  $H_{a_1,a_3,0}=\{(la_1+k,0,la_3,l,k);\ l,k \in \mathbb R \}$, $a_3>0$,  such that either \\[1mm]
 $\ {\ }a_1=1$ or $a_3<\frac{4 e^p}{(e^p-1)^2}$ and $a_1=\frac{1}{2}(p a_3+2+a_3-\frac{a_3}{e^p})$ with
$p \in \mathbb R \backslash \{ 0 \}$. \\[1mm]
{\rm (c)}  \quad $H_{a_1,a_3,b_3}=\{(la_1+k,0,la_3+kb_3,l,k);\ l,k \in \mathbb R \}$
such that for the real parameters $a_1,a_3,b_3$ one of the following conditions is satisfied:
\begin{itemize}
\item[$(\alpha )$] \quad $b_3 <0$, $b_3^2<a_3$, $a_1=b_3+1$,
\item[$(\beta )$] \quad $b_3 >0$, $b_3^2+2 b_3 \le a_3$,  $a_1=b_3+1$,
\item[$(\gamma )$] \quad  $b_3<0$, $a_3>0$, $a_1= a_3 b_3^{-1}+1$,
$b_3 (\ln \frac{b_3^2}{a_3} -1)+2+\frac{a_3}{b_3} >0$.
\end{itemize}

Any subgroup in {\rm (a), (b)} and {\rm (c)} is the stabilizer of the identity $e$ of $L$ in~$G$. No loop having  the
stabilizer of $e$  in {\rm (a)} is isotopic to a loop  having  the stabilizer in {\rm (b)}. Moreover, the loops $L_a$ and
$L_b$ corresponding to the stabilizers $H_{a,0,0}$ respectively $H_{b,0,0}$ are isomorphic if and only if $b= \pm a$.
The loops $L_{1,a_3,0}$ and $L_{1,a_3',0}$ corresponding to the stabilizers $H_{1,a_3,0}$ respectively
 $H_{1,a_3',0}$ in {\rm (b)} are isotopic precisely if $a_3=a_3'$. No loop having the stabilizer of $e$  in {\rm (c)} is
isotopic to a loop with  the stabilizer of $e$ in {\rm(a)} as well as  to a loop $L_{1,a_3,0}$. There are
infinitely many non-isotopic loops having stabilizers in {\rm (c)}.

\smallskip
Denote by $f(p)$, $g(p)$, $h(p)$, $k(p,a_1)$, $l(p,a_1)$, $n(p,a_1)$, $m(p,a_1)$ and $v(p)$   the
following  functions of the real variables $p$ and $a_1$:
\begin{align*}
& f(p)=\frac{1+e^p(1+p)}{-1+e^p(1+p)},\quad  g(p)=\frac{e^{2p}+2p e^p -1}{(e^p -1)^2},\quad
h(p)=\frac{ p^2 e^{2p}{-}\,e^{2p}{-}\,2 p e^p {+}\,1}{ p^2 e^{2p} - e^{2p}+2  e^p -1}, \\[1mm]
& k(p,a_1)=\frac{e^p(2\,{-}\,2 a_1)}{e^p(1\,{+}\,p)(a_1{-}\,1)\,{-}\,(1\,{+}\,a_1)}, \quad l(p,a_1)=e^p( a_1{-}\,1),
\quad m(p,a_1)=\frac{a_1{-}1}{1{+}p}\\[1mm]
& n(p,a_1)=\frac{ a_1(e^{p}{-}\,1)^2{-}\, e^{2p}{-}\,2p e^p {+}\,1}{p^2 e^p},
\quad v(p)=\frac{(1+p)(e^{2p}+2 p e^p-1)-p^2 e^p}{(1+p)(e^{2p}-2 e^p+1)-p^2 e^p}. \end{align*}
If a loop $L$ has a stabilizer $H$ of $e$  not contained in {\rm (a), (b)} or  {\rm (c)} then
$H=H_{a_1,a_3,b_3}=\{(la_1+k,0,la_3+kb_3,l,k)$; $l,k \in \mathbb R \}$ and there exists either a real number $p<0$ such
that one of the following conditions is satisfied:
\begin{flalign*}
& \text{\rm (i)} && f(p) < a_1 <1,\   k(p,a_1) < b_3 <0, && \\[2pt]
& \text{\rm (ii)} &&   g(p) \le a_1 <f(p), \  l(p,a_1) <b_3<0,&&  \\[2pt]
& \text{\rm (iii)}  && h(p) <a_1 < g(p),  \  l(p,a_1) <b_3< n(p,a_1), && \\[2pt]
& \text{\rm (iv)}  &&  p<-1,\  g(p) \le a_1<f(p),\  0< b_3< \min\{k(p,a_1),m(p,a_1),n(p,a_1) \} && \\[2pt]
& \text{\rm (v)}  &&  p=-1,\  g(-1) \le a_1<-1,\  0< b_3<\min \{k(-1,a_1),n(-1,a_1) \}, &&\\[2pt]
& \text{\rm (vi)} & & -1<p<0,\  g(p) \le a_1<f(p),\ 0< b_3<\min  \{k(p,a_1),n(p,a_1) \}, &&\\[2pt]
& \text{\rm (vii)} &&  p<-1,\ \ f(p)<a_1<1, \ \ 0< b_3<\min\{ k(p,a_1),n(p,a_1) \},&& \\[2pt]
& \text{\rm (viii)}  &&    p=-1,\ \ g(-1) \le a_1<1,\ \ 0< b_3<n(-1,a_1), &&\\[2pt]
& \text{\rm (ix)}  && -1<p<0,\ \ f(p)<a_1 \le 1, \ \ 0< b_3< n(p,a_1),&& \\[2pt]
& \text{\rm (x)} &&  -1<p<0,\ \ 1<a_1<v(p),\ \  m(p,a_1)<b_3 <n(p,a_1),&& \\[2pt]
& \text{\rm (xi)} &&  0< b_3< -\frac{2}{p+1-e^p},&& \end{flalign*}
or there exists a real number $p>0$ such that one of the
following conditions holds:
\begin{flalign*}
& \text{\rm(xii)} && f(p)< a_1 < g(p), \ \max \left \{ k(p,a_1), n(p,a_1) \right \}< b_3 < 0,&& \\[2pt]
& \text{\rm(xiii)} && 1 < a_1 <f(p), \  n(p,a_1) < b_3 <0,&& \\[2pt]
& \text{\rm(xiv)} && h(p)< a_1< 1, \  n(p,a_1) < b_3 < l(p,a_1), &&\\[2pt]
& \text{\rm(xv)} && 1 < a_1 < f(p),\  0<b_3<k(p,a_1), &&\\[2pt]
& \text{\rm(xvi)} && f(p)< a_1 \le g(p), \ 0 < b_3< m(p,a_1),&& \\[2pt]
& \text{\rm(xvii)} && g(p)<a_1<v(p), \  n(p,a_1)< b_3 < m(p,a_1), &&\\[2pt]
& \text{\rm(xviii)} && 0< b_3< -\frac{2}{p+1-e^p}. &&\end{flalign*}

Moreover, one has $a_3= \frac{e^p(1+p)(b_3 a_1-b_3)+e^p(2 a_1-2)-b_3-b_3 a_1}{e^p(1+p)-1}$ in the cases
{\rm (i)} till {\rm (x)} and  {\rm (xii) till (xvii)}, whereas $a_3=b_3(p b_3+b_3+2)$ and $a_1=\frac{a_3}{b_3} -1$ holds
true in the cases {\rm(xi)} and {\rm(xviii)}.

There are infinitely many non-isotopic loops $L$ having stabilizers $H_{a_1,a_3,b_3}$ such that the parameters $a_1$,
$a_3$ and $b_3$ satisfy one of the conditions {\rm (i)} till {\rm (xviii)}.

No loop for which the parameters $a_1$, $a_3$ and $b_3$ satisfy one of  {\rm (i) till (xviii)} is isotopic to a loop
corresponding to a stabilizer contained in {\rm(a)}. Moreover, no loop for which the parameters $a_1$, $a_3$ and $b_3$
satisfy one of the conditions {\rm (i) till (iii), (x)} and {\rm (xii)} till {\rm (xviii)} is isotopic to a loop having as
stabilizer $H_{1,a_3,0}$ of~{\rm (b)}.
\end{Theo}

\begin{proof}
It remains to prove the assertions concerning the isotopisms between loops having $G$ as the group topologically generated
by the left translations.

The loops $L_{a_1,a_3,b_3}$ and $L_{a_1',a_3',b_3'}$ corresponding to the pairs $(G,H_{a_1,a_3,b_3})$ and
$(G,H_{a_1',a_3',b_3'})$ are isotopic if there exists an element $g {\in} G$ such that
$g^{-1} {\bf h}_{a_1,a_3,b_3} g= {\bf h}_{a_1',a_3',b_3'}$, where ${\bf h}_{a_1,a_3,b_3}$ is the Lie algebra of the
stabilizer $H_{a_1,a_3,b_3}$. The group $G$ is the semidirect product of the $4$-dimensional normal abelian subgroups
\newline $\{(x_1,0,x_3,x_4,x_5);\ x_1,x_3,x_4,x_5 \in \mathbb R\}$ by the $1$-dimensional subgroup\newline
$\{(0,x_2,0,0,0);\ x_2 \in \mathbb R \}$.  Hence  ${\bf h}_{a_1,a_3,b_3}$ and  ${\bf h}_{a_1',a_3',b_3'}$ are conjugate
if and only if they are conjugate under an element $(0,x_2,0,0,0) \in G$. This is the case  if and only if there exists
$x_2 \in \mathbb R$ such that the following system (I) of equations
\begin{align*}
& -a_3+(a_1 a_3'+b_3'-a_1 b_3' a_1')  \sinh x_2+(a_3'+a_1 b_3'-a_1' b_3') \cosh x_2=0 \tag{1} \\
&(a_3'-a_1' b_3')  \sinh x_2 -b_3 +b_3' \cosh x_2=0 \tag{2} \\
&(a_1 -a_1'+a_3' x_2- a_1' b_3' x_2+a_1 b_3' x_2)  \cosh x_2  \\
&\quad\qquad +(1 -a_1 a_1'+  a_3' a_1 x_2+b_3' x_2 -a_1 a_1' b_3' x_2)  \sinh x_2=0 \tag{3}\\
& (1+b_3' x_2) \cosh x_2 -1 +(a_3' x_2-a_1' -a_1' b_3' x_2)\sinh x_2 =0 \tag{4} \end{align*}
has a solution. From the equation (2) we obtain that for $\sinh x_2 \neq 0$
\[
a_3'=\frac{b_3-b_3' \cosh x_2+a_1' b_3' \sinh x_2}{\sinh x_2}. \]
Putting this expression into  the equations (1),  (3) and (4) one obtains
\begin{align*}
& b_3'=-a_3 \sinh x_2-a_1 b_3 \sinh x_2 -b_3 \cosh x_2 \tag{1'} \\
& (a_1- a_1') \cosh x_2 \sinh x_2 +a_2 b_3 x_2 \sinh x_2 -1+a_1 a_1' \\
& \qquad \qquad +x_2 b_3 \cosh x_2 +(1-a_1 a_1') (\cosh x_2)^2-x_2 b_3'=0 \tag{3'} \\
& -1+\cosh x_2-a_1' \sinh x_2+x_2 b_3=0. \tag{4'} \end{align*}
The equation (4') yields for $\sinh x_2 \neq 0$ that
\[
a_1'=\frac{ \cosh x_2 +x_2 b_3 -1}{\sinh x_2}. \]
Using this expression for $a_1'$ the equation (3') reduces to
\begin{gather*}
-1+\cosh x_2- x_2 b_3' +a_1 \sinh x_2=0. \tag{3''} \end{gather*}
If we substitute  for $b_3'$ from the equation (1') in (3'') we see that
the system (I) is solvable if and only if $x_2$ is the solution of the
equation
\begin{gather*}
(a_1 b_3 x-a_3 x-a_1)(e^{2 x} -1)-(e^x-1)^2+b_3 x (e^{2x}+1)=0, \tag{i}
\end{gather*}
the parameters $b_3'$ respectively $a_1'$ satisfies  (1') respectively (4') and $a_3'=a_3$ holds.

The condition $a_3=a_3'$ yields  the  following claims: \\
No loop with stabilizer in (a) can be isotopic to a loop having
the stabilizer of $e$ not in (a). \\
The loops  $L_{1,a_3,0}$  and $L_{1,a_3',0}$ are not isotopic if $a_3 \neq a_3'$. \\
The loops having the stabilizers  $H_{b+1,b^2+1,b}$ and $H_{b'+1,b'^2+1,b'}$ with $b, b' <0$ and  $b \neq b'$ are not
isotopic. \\
Among the loops having the stabilizers $H_{a_1,a_3,b_3}$ such that the parameters $a_1$, $a_3$, $b_3$ satisfy one of the
conditions (i) till (xviii) there are infinitely many corresponding to different values of  $a_3$. Hence there are
infinitely many isotopism classes of such loops.

\smallskip
For $b_3=a_3=0$ and $0 \le a_1 <1$ the equation (i) reduces to
\[
(e^x-1)[(1+e^x)a_1+(e^x-1)]=0. \]
The solutions of this equation  are $x_2=0$ and $x_2=ln \ \frac{1-a_1}{1+a_1}$. Therefore the loop $L_{a_1}$ with the
stabilizer $H_{a_1,0,0}$ in (a) is isotopic to the loop $L_{-a_1}$ having the stabilizer $H_{-a_1,0,0}$. Since the
automorphism $\alpha $ of the Lie algebra ${\bf g}$ of $G$ given by
\[
\alpha (e_1)=-e_1, \alpha (e_5)=-e_5, \alpha (e_i)=e_i, i=2,3,4 \]
leaves the subspace $\bf m$ invariant and changes the Lie algebra ${\bf h}_{a_1,0,0}$ to
 ${\bf h}_{-a_1,0,0}$   the loops $L_{a_1}$ and $L_{-a_1}$ are already isomorphic.

\smallskip
For $b_3=0$, $a_1=1$ and $a_3>0$  the equation  (i) reduces to
\begin{gather*}
(e^x-1)[(1+e^x)(x a_3+1)+(e^x-1)]=0. \tag{ii} \end{gather*}
We consider the function
\[
f(y)=(1+e^y)(y a_3+1)+(e^y-1), \  \hbox{where} \  a_3>0. \]
For the derivations of $f(y)$ one has
\begin{align*}
f'(y)& =e^y(y a_3+a_3+2)+a_3, \\
f''(y) &=e^y(y a_3+2a_3+2), \\
f'''(y)&=e^y(y a_3+3a_3+2). \end{align*}
Since  $f''(y)=0$  only for $p=-2-\frac{2}{a_3}$ holds and $f'''(p)>0$, the function $f'(y)$ assumes in $p$ its
unique minimum. The function $f(y)$ is monotone increasing since $f'(p)=a_3(1-e^p)>0$. We have
$\lim _{y \to \infty} f(y)= \infty$ and  $\lim _{y \to -\infty} f(y)= - \infty$. Hence there is only one value $u$ for
which $f(u)=0$. Since $f(y)>0$ for all $y \ge 0$ we obtain that $u<0$ and thus the equation (ii) has precisely two
solutions $x_2=0$ and $x_2=u$. The unique loop isotopic to the loop $L_{1,a_3,0}$ corresponds to the stabilizer
$H_{a_1',a_3',b_3'}$    the parameters $a_1'$, $a_3'$, $b_3'$ of which satisfy
\[
a_3'=a_3>0,\quad  a_1'=\frac{e^u-1}{e^u+1}<0,\quad b_3'=\frac{a_3(1-e^{2u})}{2 e^u}>0. \]
But for such parameters none of the conditions $(\alpha )$, $(\beta )$, $(\gamma )$ in (c) and none of the conditions
(i) till (iii), (x) and (xii) till (xviii) is satisfied.
\end{proof}

\subsection{Bol loops corresponding to a Lie triple system which is a  non-split  extension of  its centre}

Now we treat the Lie triple systems  described in the case {\bf 2 c} in Section 3.

\begin{Lemma}
The universal Lie algebras ${\bf g}^U_{\pm}$ of the Lie triple systems ${\bf m}^{\pm}=\langle e_1,e_2,e_3\rangle$ of type
{\bf 2 c} coincide with the standard enveloping Lie algebras ${\bf g^*}_{(\pm)}$ given in {\bf 2 c}.
\end{Lemma}

\begin{proof}
Since for  ${\bf g}^U$ one has ${\bf m}^U \cap [{\bf m}^U, {\bf m}^U]=0$ we may assume that
${\bf m}^U=\langle e_1, e_2, e_3 \rangle$ and that a  basis of $[{\bf m}^U, {\bf m}^U]$ consists of  $e_4:=[e_2,e_3]$,
$e_5:=[e_1,e_3]$ and $e_6:=[e_1,e_2]$. Using the Lie triple system relations given in {\bf 2 c}  we have the following
multiplication:
\[
[e_2,e_3]=e_4, \quad  [e_4,e_2]= e_1, \quad  [e_4,e_3]=\pm e_2, \quad  [e_1,e_3]=e_5, \quad [e_1,e_2]=e_6 \]
and the other products are zero. Moreover,  one has
\begin{align*}
& [[e_2,e_3],e_4]+[[e_3,e_4],e_2]+[[e_4,e_2],e_3] = e_5 \\
& [[e_4,e_3],e_1]+[[e_3,e_1],e_4]+[[e_1,e_4],e_3] = \mp e_6. \end{align*}
Hence the Jacobi identity is satisfied if and only if $[e_1,e_3]=[e_1,e_2]=0$. From this  the assertion follows.
\end{proof}

\noindent
The  Lie groups  $G_{(+)}$ and $G_{(-)}$ corresponding to the Lie algebras
 ${\bf g}^*_{(+)}$ or   ${\bf g}^*_{(-)}$ respectively, are the semidirect products of the
$1$-dimensional Lie group
$$
 C = \left \{ \begin{pmatrix}
1 & 0 & 0 & 0 \\
0 & \bcos t & \bsin t & 0 \\
0 & \epsilon\ \bsin t & \bcos t & 0 \\
0 & 0 & 0 & 1 \end{pmatrix},\quad  t \in \mathbb R  \right \} $$
and the $3$-dimensional nilpotent Lie group
$$
B =\left \{ \begin{pmatrix}
1 & -x_2 & x_4 & x_1 \\
0 & 1 & 0 & x_4 \\
0 & 0 & 1 & x_2 \\
0 & 0 & 0 & 1 \end{pmatrix},\quad  x_1,x_2,x_4 \in \mathbb R  \right \}, $$
where the triple ($\bcos t$, $\bsin t$, $\epsilon$) denotes $(\cosh t, \sinh t, 1)$ in case  $G_{(+)}$ and ($\cos t$, $\sin t$, $-1)$ in case  $G_{(-)}$.
\newline
Denoting  the elements of  $G_{(\pm)}$  by $g(t,x_1,x_2,x_4)$ we see that
the  multiplication in $G_{(\pm)}$
is given by
\begin{align*}
& g(t_1,x_1,x_2,x_4) \cdot g(t_2,y_1,y_2,y_4) \\
&\qquad =g(t_1+t_2, x_1+y_1+ \epsilon y_4( x_2 \bcos t_2 - \epsilon x_4 \bsin t_2)
- \epsilon y_2( x_4 \bcos t_2 -x_2 \bsin t_2), \\
&y_2+x_2 \bcos t_2 - \epsilon x_4 \bsin t_2, y_4-x_2 \bsin t_2 +x_4 \bcos t_2). \end{align*}
A $1$-dimensional subalgebra ${\bf h}$ of ${\bf g}^{\ast }_{(\pm )}$  which complements
${\bf m}=\langle e_1, e_2, e_3 \rangle $, can be written as:
\[
{\bf h}=\langle e_4 + \alpha e_1 + \beta e_2 + \gamma e_3 \rangle \quad \hbox{with} \quad
\alpha , \beta , \gamma \in \mathbb R. \]
Any automorphism $\alpha$ of ${\bf g}^\ast_{(\pm )}$  leaving ${\bf m}=\langle e_1,e_2, e_3\rangle$ invariant is given  by
\[
\alpha (e_1)=\pm a^2 e_1,\quad \alpha (e_2)=\pm \epsilon  a c e_1 +a e_2,\quad
\alpha (e_3)=b e_1 +c e_2 \pm  e_3,\quad  \alpha (e_4)=\pm a e_4, \]
where $a \in \mathbb R \backslash \{ 0 \}$, $b,c \in \mathbb R$, $\epsilon =1$ in the case ${\bf g}^{\ast }_{(+)}$
and $\epsilon =-1$ for  ${\bf g}^{\ast }_{(-)}$. Using  suitable automorphisms of this form we can reduce ${\bf h}$ to one
of the following:
\[{\bf h}_1= \langle e_4 \rangle ,\quad {\bf h}_2= \langle e_4 + e_3\rangle ,\quad
 {\bf h}_{3,y}= \langle e_4 + y  e_2 \rangle , y >0, \quad {\bf h}_4= \langle e_4 +e_1 \rangle . \]
The exponential image of the subspace ${\bf m}$ has the shape
\begin{align*}
\exp {\bf m}& =\exp \{n e_1 + k e_2 +t e_3,  t,n, k \in \mathbb R \} \\
& =\Big \{ g(t,n+ \frac{k^2}{t}-\frac{k^2}{t^2} \bsin t,\frac{k}{t} \bsin t, \frac{k}{t}(1- \bcos t)),\
t,n, k \in \mathbb R\Big \} \end{align*}
(cf.\ \cite{bouetou3} p.\ 11 and p.\ 12) if we identify  ${\bf m}$ with the subspace generated by
$$
\left \langle  \left ( \begin{pmatrix}
0 & 0 & 0 & 0 \\
0 & 0 & t & 0 \\
0 & \epsilon t & 0 & 0 \\
0 & 0 & 0 & 0  \end{pmatrix},  \begin{pmatrix}
0 & -k & 0  & n  \\
0 & 0 & 0 & 0  \\
0 & 0 & 0 & k \\
0 & 0 & 0 & 0  \end{pmatrix} \right) \right \rangle. $$
First we investigate the group  $G_{(-)}$. The element  $g(\frac{\pi}{2},0,1,0) \in G_{(-)}$ conjugates
$\exp  e_4 \in H_1$  to  $\exp(-2 e_1) \in \exp {\bf m}$ and $\exp (e_4+e_1)\in H_4$ to $\exp(-e_2-e_1) \in \exp\ {\bf m}$.
Moreover,  $\exp \pi (e_4+e_3) \in H_2$ is conjugate to  $\exp \pi (e_1+e_3) \in \exp {\bf m}$ under
$g(0,0,-1,0) \in G_{(-)}$ and  $\exp  (e_4+y e_2) \in H_{3,y}$ is for all $y \in \mathbb R \backslash \{ 0 \}$
conjugate to $\exp [( \sin \arc \ctg y)^{-1} e_2] \in \exp {\bf m}$ under  $g(-\arc \ctg y ,0,0,0) \in G_{(-)}$.
Hence there is no $3$-dimensional differentiable Bol loop $L$ such that the
group topologically generated by its left translations is the Lie group $G_{(-)}$ (cf.\ Lemma~1).

\smallskip
Finally we deal with the group  $G_{(+)}$. The element $\exp(e_4+e_3)\in H_2$ is conjugate to
$\exp (e_3-e_1)$ of $\exp {\bf m}$ under  $g(0,0,1,0) \in G_{(+)}$. The element $\exp l(e_4+y e_2) \in H_{3,y}$  with
$l=-\sinh \big(\frac{1}{2} \ln \frac{y-1}{y+1}\big)$  is conjugate to $\exp e_2 \in \exp {\bf m}$ under
$g \big(\frac{1}{2}\ln \frac{y-1}{y+1},0,0,0 \big) \in G_{(+)}$ for all $y>1$. Therefore we may suppose that the
stabilizer of the identity of a Bol loop $L$ is either the Lie group $H_1$ or $H_4$ or  $H_{3,y}$, where $0<y \le 1$.

\smallskip
Each element $g \in G_{(+)}$ can  be represented  uniquely as a product $g=m h$, where  $m \in \exp {\bf m}$ and $h$ is an
element of $H_1$, $H_4$ or $H_{3,y}$ with  $0<y \le 1$ respectively, if and only if
for given $t_1,x_1, x_2,x_4 \in \mathbb R$ the equation
\[
g(t_1,x_1,x_2,x_4)=g\Big(t,n+\frac{k^2}{t}-\frac{k^2}{t^2} \sinh t,\frac{k}{t} \sinh t, \frac{k}{t}(1-\cosh t)\Big)\cdot h\]
is uniquely solvable for
\[
h=g(0,0,0,a) \in H_1,\   h=g(0,a,0,a) \in H_4,\ \hbox{and} \ h= g(0,0,l y,l) \in H_{3,y}. \]
In the case of $H_1$ the unique solution is given by:
\begin{gather*}
t :=t_1, \quad  k:= \frac{x_2}{\frac{\sinh t_1}{t_1}}, \quad a:=x_4 - \frac{x_2 (1-\cosh t_1)}{\sinh t_1}, \\
n:=x_1-x_4 x_2 + \frac{x_2^2 (1-\cosh t_1)}{\sinh t_1}- \frac{x_2^2 (t_1-\sinh t_1)}{\sinh ^2 t_1}. \end{gather*}
In the case of $H_4$ we obtain as unique solution
\begin{gather*}
t:=t_1, \quad  k:= \frac{x_2}{\frac{\sinh t_1}{t_1}}, \quad a:=x_4 - \frac{x_2 (1-\cosh t_1)}{\sinh t_1}, \\
n:=x_1 -(1+x_2) \left[ x_4- \frac{x_2 (1-\cosh t_1)}{\sinh t_1}\right]
  - \frac{x_2^2 (t_1-\sinh t_1)}{\sinh ^2 t_1}. \end{gather*}
Moreover, in the case of  $H_{3,y}$ for  $y \in (0,1]$ the unique solution is given  as follows:
\newline
For $t_1=0$ we have $t=0$, $l=x_4$, $k=x_2-x_4 y$, $n=x_1-x_4(x_2-y x_4)$,
\newline
whereas for $t_1 \neq 0$ we obtain
\begin{gather*}
t=t_1, l= \frac{x_4 \sinh t_1 +x_2 \cosh t_1 -x_2}{\sinh t_1 -y +y \cosh t_1},
 \quad  k= \frac{(x_2 - y x_4) t_1}{\sinh t_1 -y +y \cosh t_1}, \\[2mm]
n=x_1 +  \frac{(\sinh t_1 - t_1)(x_2 -y x_4)^2}{(\sinh t_1 -y +y \cosh t_1)^2}
- \frac{(x_4 \sinh t_1 +x_2 \cosh t_1 -x_2)(x_2 -y x_4)}{(\sinh t_1 -y +y \cosh t_1)}. \end{gather*}
It follows that the group $G_{(+)}$  is the group topologically generated by the left translations of infinitely many
non-isomorphic  differentiable $3$-dimensional Bol loops $L$. Every such loop $L$  has a normal subgroup
$N=\exp\{\lambda e_1,\ \lambda \in \mathbb R \}=\{g(0,\lambda ,0,0),\ \lambda \in \mathbb R \}$ isomorphic to $\mathbb R$
and the factor loop $L/N$  is isomorphic to a loop $L_{\alpha }$  with $\alpha \le -1$ defined in Theorem 23.1 of
\cite{loops} and thus isotopic to the pseudo-euclidean plane loop. Hence $L$  is an  extension of the
group $\mathbb R$ by a loop~$L_{\alpha }$.

The loop $L_1$ having  $H_1$ as  the stabilizer of $e \in L_1$ in $G_{(+)}$ is a  Bruck loop. The loop $L_2$
which is realized on the factor space $G/H_4$ is a left A-loop. The stabilizer $H_1$ is conjugate to  $H_4$ under
$g(0,0,-\frac{1}{2},0) \in G_{(+)}$ and to $H_{3,y}$ under $g(\text{artanh} (-y),0,y,1) \in G_{(+)}$ with $y \in
(0,1)$. Hence the loops corresponding to these stabilizers are isotopic. In contrast to this the loop
corresponding to $H_{3,1}= \{ g(0,0,l,l)$; $ l \in \mathbb R \}$ does not belong to the isotopism class of $L_1$.

These considerations yield  the following

\begin{Theo}
If $L$ is a $3$-dimensional connected differentiable Bol loop corresponding to a Lie triple system, which is a non-split
extension of its centre and a $2$-dimensional non-abelian Lie triple system, then the group $G$ topologically generated
by the left translations of $L$ is the  semidirect product of the normal group $\mathbb R$ and  the $3$-dimensional
non-abelian nilpotent Lie group such that the multiplication of $G$ is given by
\begin{align*}
& g(t_1,x_1,x_2,x_4) \cdot g(t_2,y_1,y_2,y_4) \\
&\quad =g(t_1+t_2, x_1+y_1+  y_4( x_2 {\cosh t_2} -  x_4 {\sinh t_2})-  y_2( x_4 {\cosh t_2} -x_2 {\sinh t_2}), \\
& y_2+x_2 {\cosh t_2} -  x_4 {\sinh t_2}, y_4-x_2 {\sinh t_2} +x_4 {\cosh t_2}). \end{align*}
All loops $L$  are extensions of the Lie group $\mathbb R$ by a loop $L_{\alpha }$ described in  Theorem~23.1
of~\cite{loops} and form precisely two isotopism classes ${\cal C}_1$, ${\cal C}_2$.

All  loops  in  ${\cal C}_1$ are  isomorphic and may be represented by the loop $L$ which has
the group  $H=\{ g(0,0,l,l);\ l \in \mathbb R \}$ as the stabilizer of its identity in $G$.

The class  ${\cal C}_2$ contains (up to isomorphisms) a  Bruck loop $L_1$  corresponding to
$H_1=\{ g(0,0,0,a),\ a \in \mathbb R \}$, a  left A-loop $L_2$  corresponding to $H_2=\{ g(0,a,0,a),\ a \in \mathbb R\}$
and the loops $L_y$ with $y \in (0,1)$ corresponding to the groups $H_{y}=\{ g(0,0,ly,l),\ l \in \mathbb R \}$  as the
stabilizers of the identity.
\end{Theo}

\section{Bol loops corresponding to the Lie triple system\\ having  trivial centre}

Now we deal with the case {\bf 3} in Section 3.

\begin{Lemma}
The universal Lie algebra ${\bf g}^U$ of the Lie triple system ${\bf m}=\langle e_1, e_2, e_3 \rangle$ of type {\bf 3}
is the standard enveloping Lie algebra ${\bf g^*}$ characterized in {\bf 3} of  Section~3.
\end{Lemma}

\begin{proof}
Because  of  ${\bf m}^U \cap [{\bf m}^U, {\bf m}^U]=0$ we may assume that ${\bf m}^U\!=\langle e_1, e_2, e_3 \rangle$ and
take for a  basis of $[{\bf m}^U, {\bf m}^U]$ the vectors  $e_4:=[e_2,e_3]$, $e_5:=[e_1,e_3]$ and $e_6:=[e_1,e_2]$.
The relations of the Lie triple system of type {\bf 3} yield  the following multiplication:
\[
[e_2,e_3]=e_4, \quad [e_4,e_3]= e_1, \quad  [e_1,e_3]=e_5, \quad [e_1,e_2]=e_6, \]
whereas the other products are zero. For $e_2$, $e_3$, $e_4$ one has
\[
[[e_2,e_3],e_4]+[[e_3,e_4],e_2]+[[e_4,e_2],e_3]=e_6 \]
and the Jacobi identity is satisfied   if and only if  $[e_1,e_2]=0$. This is the assertion.
\end{proof}

The mapping $\beta $
\begin{align*}
\beta(e_1)& = \frac{1}{2} \sqrt{2} e_1- \frac{1}{2} \sqrt{2} e_4-\frac{1}{2} \sqrt{2} e_2+\frac{1}{2} \sqrt{2} e_5, \\
\beta(e_2)& = \frac{1}{2} \sqrt{2} e_1- \frac{1}{2} \sqrt{2} e_4+\frac{1}{2} \sqrt{2} e_2-\frac{1}{2} \sqrt{2} e_5, \\
\beta(e_3)& = \frac{1}{2} \sqrt{2} e_3, \quad \beta(e_4)=e_1+e_4, \quad \beta(e_5)=-e_2-e_5 \end{align*}
yields an isomorphism of ${\bf g}^*$ onto the Lie algebra ${\bf g}$ defined  by the following non-trivial products:
\[
[e_1,e_3]= e_1 -e_2, \quad [e_2,e_3]=e_1+e_2, \quad [e_4,e_3]=- e_4+e_5, \quad [e_5,e_3]=-e_5-e_4. \]
(We remark, that ${\bf g}$ is isomorphic to the Lie algebra $g_{5,17}$ for $s=-1$, $q=-1$, $p=1$ in
\cite{mubarak} (p.\ 105)). The elements $x e_1+y e_2+z e_3+u e_4+v e_5$ of  ${\bf g}$ can be identify with the matrices
\[
\begin{pmatrix}
0 & y & x & 0 & 0 & 0 \\
0 & z & z &  0 & 0 & 0 \\
0 & -z &  z & 0 & 0 & 0 \\
0 & 0 & 0 & 0 & v & u \\
0 & 0 & 0 & 0 & -z & -z \\
0 & 0 & 0 & 0 & z & -z \end{pmatrix};\quad x,y,z,u,v \in \mathbb R . \]
Then the multiplication in $G$ is determined  by
\begin{align*}
& g(a_1,b_1,c_1,d_1,f_1) g(a_2,b_2,c_2,d_2,f_2) \\
&\qquad =g(a_2+b_1 e^{c_2} \sin{c_2}+a_1 e^{c_2} \cos{c_2}, b_2 +b_1 e^{c_2}  \cos{c_2}-a_1  e^{c_2} \sin{c_2}, \\
& c_1+c_2, d_2 -f_1 e^{-c_2} \sin{c_2}+d_1 e^{-c_2}  \cos{c_2},
f_2+f_1 e^{-c_2}  \cos{c_2}+d_1 e^{-c_2} \sin{c_2}). \end{align*}
The isomorphism $\beta $ maps the Lie triple system $\langle e_1,e_2,e_3 \rangle $ onto the Lie triple system
${\bf m}=\langle e_1-e_4,e_2-e_5,e_3 \rangle$ and one has
\begin{align*}
\exp {\bf m}& =\exp \{n(e_1-e_4)+m(e_2-e_5)+s e_3; n,m,s \in \mathbb R \} \\[1mm]
& =\left \{ g \left( \frac{(n-m)(e^{s} \cos{s}-1)+(n+m)e^{s} \sin{s}}{2 s},\right.\right.\\[1mm]
&\qquad \frac{(n+m)(e^{s} \cos{s}-1)+(m-n) e^{s} \sin{s}}{2 s},s, \\[1mm]
& \qquad \frac{(n-m)(e^{-s} \cos{s}-1)- (m+n) e^{-s} \sin{s}}{2 s}, \\[1mm]
&\qquad \left.\left. \frac{(n+m)(e^{-s}\cos{s}-1)+(n-m)e^{-s} \sin{s}}{2 s} \right),\ n,m,s \in \mathbb R \right \}.
\end{align*}
The $2$-dimensional subalgebras ${\bf h}$ of ${\bf g}$ with the property
${\bf h} \cap \langle e_3 \rangle=\{ 0 \}$ have the following forms:
\begin{align*}
{\bf h}_{a_2,a_4,b_2}& =\langle e_5+a_2 e_2+a_4 e_4, e_1+b_2 e_2 \rangle && \hbox{with}  \  a_2,a_4,b_2 \in \mathbb R, \\
{\bf h}_{a_1,a_4,b_1}&=\langle e_5+a_1 e_1+a_4 e_4, b_1 e_1+e_2 \rangle, && \hbox{where} \ a_1,a_4,b_1 \in \mathbb R, \\
{\bf h}_{a_1,a_2,b_1,b_2}&=\langle e_5+a_1 e_1+a_2 e_2, e_4+b_1 e_1+b_2 e_2 \rangle, &&
\hbox{where} \ a_1,a_2,b_1,b_2 \in \mathbb R. \end{align*}
The automorphism  $\alpha: {\bf g} \to {\bf g}$ given by
\begin{gather*}
\alpha(e_1)=b_2 e_1+e_2,\quad \alpha(e_2)=-e_1+b_2 e_2,\quad \alpha(e_3)=e_3, \\*
\alpha(e_4)=b_2 e_4-e_5,\quad  \alpha(e_5)=e_4+b_2 e_5, \end{gather*}
where $b_2 \in \mathbb R$, and the automorphism $\beta: {\bf g} \to {\bf g}$
determined by
\begin{gather*}
\beta(e_1)=e_1+b_1 e_2,\quad \beta(e_2)=-b_1 e_1+e_2,\quad \beta(e_3)=e_3, \\
\beta(e_4)= e_4-b_1 e_5,\quad  \beta(e_5)=b_1 e_4+e_5 \end{gather*}
where $b_1\in \mathbb R$, leave the subspace ${\bf m}$ invariant. If $b_2\neq a_4$ then $\alpha$ maps
${\bf h}_{a_2,a_4,b_2}$ onto
\[
{\bf h}_{a,b}= \langle e_5+a e_1+b e_4, e_2 \rangle  \quad  \hbox{with} \quad a,b \in \mathbb R \]
and if $b_2=a_4$ then  $\alpha $ reduces ${\bf h}_{a_2,a_4,b_2}$ to
\[
{\bf h}_a= \langle e_4+a e_1, e_2 \rangle   \quad  \hbox{with} \quad a \in \mathbb R. \]
For $b_1 \neq \frac{1}{a_4}$ the automorphism  $\beta $ maps  ${\bf h}_{a_1,a_4,b_1}$ to  ${\bf h}_{a,b}$,
whereas  for  $b_1= \frac{1}{a_4}$ the subalgebras  ${\bf h}_{a_1,a_4,b_1}$ reduce to  ${\bf h}_a$.
Since  ${\bf h}_{a,b} \cap {\bf m}$ is not trivial if  $a=-b$  we may assume  that for ${\bf h}_{a,b}$ one has $a \neq -b$.

For $a_1=a_2=b_1=b_2=0$ the subalgebra ${\bf h}_{0,0,0,0}=\langle e_5, e_4 \rangle$ is an ideal of ${\bf g}$.
Therefore we suppose  that in ${\bf h}_{a_1,a_2,b_1,b_2}$ not all parameters $a_1,a_2,b_1,b_2$ are $0$. Moreover,
$(a_2+1)(1+b_1)-a_1 b_2 \neq 0$, since otherwise ${\bf h}_{a_1,a_2,b_1,b_2} \cap {\bf m} \neq 0$.

The Lie groups  corresponding to  the Lie algebras ${\bf h}_a$, ${\bf h}_{a,b}$, ${\bf h}_{a_1,a_2,b_1,b_2}$ have the forms
\begin{itemize}
\item[a)] \quad $H_a=\exp {\bf h}_a=\{ g(k a, l,0,k,0);\ k,l \in \mathbb R \}$, \ $a \in \mathbb R$
\item[b)] \quad $H_{a,b}=\exp {\bf h}_{a,b}=\{ g(k a, l,0,kb,k);\ k,l \in \mathbb R \}$, \ $a,b \in \mathbb R$, $a \neq b$
\item[c)] \quad
$H_{a_1,a_2,b_1,b_2}=\exp {\bf h}_{a_1,a_2,b_1,b_2}=\{ g(k a_1+l b_1,k a_2+l b_2,0,l,k);\ k,l \in \mathbb R \}$,
\end{itemize}
where  $(a_2+1)(1+b_1)-a_1 b_2 \neq 0$ and  not all $a_1$, $a_2$, $b_1$, $b_2$ are equal $0$.

Each element of $G$ has a unique decomposition as
\begin{align*}
g(x_1,x_2,x_3,x_4,x_5) & =g(y_1,0,y_2,0,y_3) g(k a,l,0,k,0) \ \hbox{in\ the\ case\ a)} \\
g(x_1,x_2,x_3,x_4,x_5)&=g(y_1,0,y_2,y_3,0) g(k a,l,0,k b,k) \ \hbox{in\ the\ case\ b)} \\
g(x_1,x_2,x_3,x_4,x_5)&=g(y_1,y_3,y_2,0,0)
g(k a_1+l b_1, k a_2+l b_2, 0, l,k); l,k \in \mathbb R\} \\
&\quad \ \hbox{in\ the\ case\ c)}. \end{align*}

A differentiable Bol loop $L$  exists precisely if in the case a) every element $g(y_1,0,y_2,0,y_3)$, in the case b) every
element $g(y_1,0,y_2,y_3,0)$ and in the case c) every element $g(y_1,y_3,y_2,0,0)$, $y_i \in \mathbb R$, $i=1,2,3$, can
be written uniquely as a product $g=m h$ or equivalently $m=g h^{-1}$, where $m \in \exp {\bf m}$ and $h$ is a suitable
element of the stabilizer $H_a$, $H_{a,b}$ or $H_{a_1,a_2,b_1,b_2}$ respectively. This  happens if and only if for given
$y_1,y_2,y_3 \in \mathbb R$ the following system of equations
\begin{gather*}
s=y_2,\quad A=\frac{u(e^s \cos{s}-1)+v e^s \sin{s}}{2s}, \quad B=\frac{v(e^s \cos{s}-1)-u e^s \sin{s}}{2s}, \\
C=\frac{u(e^{-s} \cos{s}-1)-v e^{-s} \sin{s}}{2s}, \quad
 D=\frac{v(e^{-s} \cos{s}-1)+u e^{-s} \sin{s}}{2s}, \tag{I} \end{gather*}
with $A=y_1-k a$, $B=-l$, $C=-k$, $D=y_3$ in the case a), \\[2pt]
with $A=y_1-k a$, $B=-l$, $C=y_3-k b$,\ $D=-k$ in the case b) and \\[2pt]
$A=y_1-k a_1-l b_1$,\  $B=y_3-k a_2-l b_2$, $C=-l$,\ $D=-k$, in the case c) \\[2pt]
has a unique solution $(u,v,s,k,l) \in \mathbb R^5$.

Assuming  $y_2 \neq 0$ and putting
$$
\begin{array}{ll}
m_{11} = e^{y_2} \cos{y_2} -1-a (e^{-y_2} \cos{y_2}-1),\qquad & m_{21} = e^{-y_2} \sin{y_2}, \\
m_{12} = e^{y_2} \sin{y_2}+ae^{-y_2} \sin{y_2} , & m_{22} = e^{-y_2} \cos{y_2}-1 \end{array}$$
in the case a),
$$
\begin{array}{ll}
m_{11}  = e^{y_2} \cos{y_2} -1+a e^{-y_2} \sin{y_2}, \quad &
m_{12}  =  e^{y_2} \sin{y_2}-a (e^{-y_2} \cos{y_2}-1), \\
m_{21}  =  e^{-y_2} \cos{y_2} -1- b e^{-y_2} \sin{y_2}, &
m_{22}  =  -e^{-y_2} \sin{y_2} -b (e^{-y_2} \cos{y_2}-1) \end{array}$$
in the case b) and
\begin{align*}
m_{11} & =  e^{y_2} \cos{y_2}+1-a_1 e^{-y_2} \sin{y_2} -b_1 (e^{-y_2} \cos{y_2}-1), \\
m_{12} & =  e^{y_2} \sin{y_2}-a_1(e^{-y_2} \cos{y_2} -1)+b_1 e^{-y_2} \sin{y_2}, \\
m_{21} & =  -e^{y_2} \sin{y_2-a_2 e^{-y_2} \sin{y_2} -b_2 (e^{-y_2} \cos{y_2}-1)}, \\
m_{22} & =  e^{y_2} \cos{y_2}+1-a_2(e^{-y_2} \cos{y_2} -1)+b_2 e^{-y_2} \sin{y_2} \end{align*}
in the case c), we see that the system (I) yields  the following system of linear equations
\begin{gather*}\abovedisplayskip=0pt
m_{11} u+m_{12} v = 2 y_1 y_2 \\
m_{21} u+m_{22} v = 2 y_2 y_3. \tag{II} \end{gather*}
If $y_1=y_3=0$ and $\det{(m_{ij})}  = 0$ $i,j \in \{ 1,2 \}$ then the
system (II) has infinitely many solutions.

The condition $\det{ (m_{ij}) } = 0$ holds if and only if in the case a) the function
\[
f(x)=-(e^{x}+e^{-x}) \cos{x} -a(e^{-2 x}-2 e^{-x} \cos{x}+1)+ 2 \cos^2{x}, \]
in the case b) the function
\begin{align*}
g(x)& =(2 a e^{-2 x}-2 b) \cos^2{x}-(2+2 a b e^{-2 x}) \cos{x} \sin{x} +b e^{x} \cos {x} \\*
&\quad +(b-2 a) e^{-x} \cos{x}+e^{x} \sin{x}+(2 a b +1) e^{-x} \sin{x} +a -a e^{-2x} \end{align*}
and in the case c) the function
\begin{align*}
h(x)&=e^{2x}+e^{-2x}(b_1 a_2-a_1 b_2)+ (e^{x}+e^{-x}) \sin{x}(a_1-b_2) \\
&\quad +e^{x} \cos{x}(a_2+b_1-2)+e^{-x} \cos{x}(2 a_1 b_2-2 a_2 b_1 +b_1+a_2) \\
&\quad +1+(2 b_2-2 a_1) \sin{x} \cos{x}-(2 b_1+2 a_2) \cos^2 {x}+ b_1 a_2-b_2 a_1 \end{align*}
assumes the value $0$.

\smallskip
If $k=\max \{ 100, 2 |a| \}$ then  for  $x=2 \pi k$ and $y= \pi+ 2 \pi k$ we obtain that  $f(x)<0$ and  $f(y)>0$. Hence
in the open  interval $\big (2 \pi k, \pi  +2 \pi k \big )$  there is a value $y_2$ such that $f(y_2)=0$.

\smallskip
For $p_1=\frac{\pi }{2}+2 \pi k$ and $p_2=\frac{3 \pi }{2}+ 2 \pi k$ with  $k=\max\{100,2|a|,4 |a b|\}$ one has $g(p_1)>0$
and  $g(p_2)<0$. Hence  the open interval $(\frac{\pi }{2}+2 \pi k,\frac{3 \pi }{2}+ 2 \pi k)$  contains a value $y_2$
such that $g(y_2)=0$.

Therefore there is no  $3$-dimensional differentiable Bol loop $L$ such that the group topologically generated by its
left translations is the group $G$ and the stabilizer of $e \in L$ in $G$ is a  subgroup $H_a$ or  $H_{a,b}$.

\smallskip
In the case c) one has
\begin{itemize}
\item[a)] \quad $\lim _{x \to + \infty} h(x)=+ \infty$,
\item[b)] \quad $\lim _{x \to - \infty} h(x)=- \infty$ if $b_1 a_2-a_1 b_2<0$
\item[c)] \quad $\lim _{x \to - \infty} h(x)= \infty$ if $b_1 a_2-a_1 b_2>0$.
\end{itemize}
The first and second derivative  of $h(x)$ are
\begin{align*}
h'(x)&=2 e^{2x}+(a_1-b_2)[(e^x-e^{-x}) \sin{x}+(e^x+e^{-x}) \cos{x}] \\
 &\quad +(a_2+b_1-2) (e^x \cos{x} -e^x \sin{x})-2e^{-2x}(b_1 a_2-a_1 b_2) \\
 &\quad -(b_1+a_2-2 a_2 b_1+2 a_1 b_2)(e^{-x} \cos{x}+e^{-x} \sin{x}) \\
 &\quad +(2 b_2-2 a_1)(\cos^2{x}-\sin^2{x})+4 \cos{x} \sin{x}(b_1+a_2) \end{align*}
and
\begin{align*}
h''(x)&=4 e^{2x}+4e^{-2x}(b_1 a_2-a_1 b_2)+2(a_1-b_2)(e^x-e^{-x}) \cos{x} \\
 &\quad +2(b_1+a_2-2 a_2 b_1+2 a_1 b_2)e^{-x} \sin{x}-2(a_2+b_1-2) e^x \sin{x} \\
 &\quad +4 (b_1+a_2)(\cos^2{x}-\sin^2{x})  -8(b_2-a_1) \cos{x} \sin{x}. \end{align*}
One obtains $h(0)=h'(0)=0$ and $h''(0)=4+4(b_1+a_2)+4(a_2 b_1-a_1 b_2)$. Since $h''(0) \neq 0$ we have two
possibilities: $h''(0)<0$ or $h''(0)>0$. The function $h(x)$ has in $0$ a maximum or a minimum according as
$h''(0)<0$ or  $h''(0)>0$. Now from the properties a) and b) it follows that for $h''(0)<0$ and for $h''(0)>0$ with
$b_1 a_2-a_1 b_2 < 0$ there is a   value $p \in \mathbb R \backslash \{ 0 \}$  such that $h(p)=0$.

For $b_1 a_2-a_1 b_2= 0$ one has
\begin{align*}
h(x)& =e^{2x}+ (e^{x}+e^{-x}) \sin{x}(a_1-b_2)+e^{x} \cos{x}(a_2+b_1-2) \\
&\quad +e^{-x} \cos{x}(b_1+a_2)+1+(2 b_2-2 a_1) \sin{x} \cos{x}-(2 b_1+2 a_2) \cos^2 {x}. \end{align*}
First we assume that $a_1-b_2 \neq 0$. Then we have $\varepsilon h(p_1) >0$ and $\varepsilon h(p_2) <0$ if
$p_1=- (\frac{\pi}{2}+2 \pi k)$ and $p_2=-(\frac{3 \pi}{2}+2 \pi k )$, where $k=\max \big \{100,\frac{4}{|a_1-b_2|}\big \}$
and $\varepsilon =1$ if $a_1-b_2 <0$, whereas $\varepsilon =-1$ for  $a_1-b_2 >0$. Hence in the open interval
$\big (-\frac{3 \pi}{2}-2 \pi k, -\frac{\pi}{2}-2 \pi k \big )$ the function $h$ assumes $0$.

For $a_1=b_2$ we obtain
\[
h(x)=e^{2x}+(b_1+a_2)(e^{x}+e^{-x}) \cos{x} -2 e^{x} \cos{x}+1 -(2 b_1+2 a_2) \cos^2 {x}. \]
If $p_1=-2 \pi k$ and $p_2=- \pi-2 \pi k$ then we have $\varepsilon h(p_1) >0$
and $\varepsilon h(p_2) <0$, where $k=\max \big \{100,\frac{4 |1-2b_1-2a_2|}{|b_1+a_2|} \big \}$ and
$\varepsilon =1$ or $\varepsilon =-1$ according as  $b_1+a_2>0$ or   $b_1+a_2<0$. Therefore the interval
$(- \pi-2 \pi k, -2 \pi k)$ contains a value $p \in \mathbb R \backslash \{ 0 \}$ such that $h(p)=0$.

It follows that a differentiable  Bol loop $L$ does not  exist if  the parameters $a_1$, $a_2$, $b_1$, $b_2$ satisfy either
\[
1+b_1+a_2+a_2 b_1-a_1 b_2 <0 \]
or
\[
1+b_1+a_2+a_2 b_1-a_1 b_2 >0 \quad  \hbox{and} \quad  a_2 b_1-a_1 b_2 \le 0 . \]
For $y_2=0=s$ the system (I) reduces to
\begin{gather*}
y_1-k a_1-l b_1=n,\ y_3-k a_2-l b_2=m,\ l =n,\  k =m, \quad  \hbox{with} \  n,m \in \mathbb R. \tag{III} \end{gather*}
Since for the parameters $a_1$, $a_2$, $b_1$, $b_2$ one has $(a_2+1)(1+b_1)-a_1 b_2 \neq 0$ the system (III)  has
precisely one solution for all $y_1,y_3 \in \mathbb R$. Namely, if $b_1 \neq -1$ we obtain
\[
l=n=\frac{y_1-m a_1}{1+b_1}, \quad  k=m=\frac{y_3(1+b_1)-f_1 b_2}{(a_2+1)(1+b_1)-a_1 b_2} , \]
whereas for $b_1=-1$ one has $ a_1 b_2 \neq 0$ and
\[
k=m=\frac{y_1}{a_1},\quad l=n=\frac{y_3 a_1-y_1 a_2-y_1}{b_2 a_1}. \]
The above discussion  yields the following

\begin{Theo}
If $L$ is a $3$-dimensional  connected differentiable Bol loop corresponding to a Lie triple system which has trivial
centre, then the group topologically generated by its left translations is the $5$-dimensional Lie group $G$ the
multiplication of which is given by
\begin{align*}
& g(a_1,b_1,c_1,d_1,f_1) g(a_2,b_2,c_2,d_2,f_2)  \\
&\qquad =g(a_2+b_1 e^{c_2} \sin{c_2}+a_1 e^{c_2} \cos{c_2}, b_2 +b_1 e^{c_2}  \cos{c_2}-a_1  e^{c_2} \sin{c_2}, \\
& c_1+c_2,d_2-f_1 e^{-c_2}\sin{c_2}+d_1 e^{-c_2} \cos{c_2},f_2+f_1 e^{-c_2} \cos{c_2}+d_1e^{-c_2} \sin{c_2}). \end{align*}
Moreover, the stabilizer of the identity of $L$ in $G$ is the subgroup
\[
H_{a_1,a_2,b_1,b_2}=\{ g(k a_1+l b_1, k a_2+l b_2, 0, l, k);\ k,l \in \mathbb R \} \] such that  the parameters
$a_1,a_2,b_1,b_2$ satisfy
\[
1+b_1+a_2+a_2 b_1-a_1 b_2 >0 \quad \hbox{and} \quad  a_2 b_1-a_1 b_2>0 \]
and the function
\begin{align*}
h(x) & =e^{2x}+e^{-2x}(b_1 a_2-a_1 b_2)+ (e^{x}+e^{-x}) \sin{x}(a_1-b_2) \\
&\quad +e^{x} \cos{x}(a_2+b_1-2)+e^{-x} \cos{x}(2 a_1 b_2-2 a_2 b_1 +b_1+a_2) \\
&\quad +1+(2 b_2-2 a_1) \sin{x} \cos{x}-(2 b_1+2 a_2) \cos^2 {x}+ b_1 a_2-b_2 a_1 \end{align*}
is positive for all $x \in \mathbb R \backslash \{ 0 \}$.
\end{Theo}

\smallskip
There are many differentiable $3$-dimensional Bol loops on the factor space $G/H_{a_1,a_2,b_1,b_2}$.
For instance  choosing  $a_1,a_2,b_1,b_2 \in \mathbb R$ such that
$a_1=b_2$, $a_2=2-b_1$  and  $c=b_1 a_2-a_1 b_2=-(b_1-1)^2-b_2^2+1$ with   $\frac{3}{7}<c \le 1$ the function
\[
h(x)=e^{2x}+(2-2c) e^{-x} \cos{x} +c e^{-2x}+1+c-4 \cos^2{x} \]
of Theorem 11 is positive for all $x \in \mathbb R\backslash \{ 0 \}$. To prove this it is enough to show that the function
\[
k(x)=e^{2x}+(2-2c) e^{-x} \cos{x} +c e^{-2x}+c-3 \]
is positive for all $x \in \mathbb R \backslash \{ 0 \}$. The second derivative
\[
k''(x)=4 e^{2x}+4 c e^{-2x}+4(1-c) e^{-x} \sin{x} \]
is positive  if and only if
\[
4 e^{2x}+4 c e^{-2x}-4(1-c) e^{-x} >0 \]
or
\[
l(x)= e^{4x}+ (c-1) e^{x}+c >0 \quad \hbox{for\ all}\ x \in \mathbb R. \]
For the derivations of  $l(x)$ we obtain
\[
l'(x)=4e^{4x}+ (c-1) e^{x}, \quad l''(x)=16e^{4x}+ (c-1) e^{x}. \]
One has $l'(p)=0$ if and only if  $p=\frac{1}{3} \ln{\frac{1-c}{4}}$. For this value $p$ the function $l(x)$ takes its
unique minimum since $l''(p)=\big (\frac{1-c}{4} \big )^{\frac{1}{3}}(3-3c)>0$. Because of $\frac{3}{7}<c \le 1$ we get
$l(p)=c-3 \big (\frac{1-c}{4} \big )^{\frac{4}{3}} \ge c-\frac{3}{4}(1-c)>0$. It follows $k''(x)>0$ for all
$x \in \mathbb R $ and therefore $k'(x)$ is a strictly monotone increasing function. Since
$k'(0)=0$ the value $0$ is the unique minimum of $k(x)$. Furthermore one has $k(x) \ge 0$ because of
$k(0)=0$ and   $\lim _{x \to -\infty} k(x)=\lim_{x \to + \infty} k(x)= + \infty$.

\smallskip
Let  $L_{a_1,a_2}$ be the Bol loop belonging to the triple $(G,H_{a_1,a_2,2-a_2,a_1}, \exp {\bf m})$, where
$-\frac{4}{7}<-(a_2-1)^2-a_1^2 \le 0$.  Among these  loops  only the loop   $L_{0,1}$ is a left A-loop.
Since there is no element $g \in G$ such that $g^{-1} {\bf h}_{a_1,a_2,2-a_2,a_1} g= {\bf h}_{a_1',a_2',2-a_2',a_1'}$
for two different pairs $(a_1,a_2)$, $(a_1',a_2')$ holds  the loops  $L_{a_1,a_2}$ and $L_{a_1',a_2'}$ are not
 isotopic. Therefore  there are infinitely many non-isotopic Bol loops $L_{a_1,a_2}$.

Author's address: \\
    Mathematisches Institut der Universit\"at Erlangen-N\"urnberg\\
    D-91054 Erlangen, Bismarckstr. 1 $\frac{1}{2}$\\
    Germany, figula@mi.uni-erlangen.de \\

    Department of Mathematics\\
    University of Debrecen\\
    H-4010 Debrecen, P.O. Box 12\\
    Hungary, figula@math.klte.hu


\begin{thebibliography}{37}

\bibitem{bouetou1} T. B. Bouetou, Classification of solvable Lie triple systems of dimension $3$, M. Sc. Thesis, Moscow Friendship of Nations University, Moscow, 1992, Russian. 


\bibitem{bouetou2} T. B. Bouetou, On Bol algebras, Webs and Quasigroups,  Tver State University, 1995, 75--83. 


\bibitem{bouetou3} T. B. Bouetou, On the classification of one class of three-dimensional Bol algebras with solvable
enveloping Lie algebras of small dimension I, arXiv:math. AG/0309122. 

\bibitem{figula} \'A. Figula, $3$-dimensional Bol loops as sections in non-solvable Lie groups, Forum Math. 17, 2005, 
431--460. 

\bibitem{locally} P. M. D. Furness, Locally symmetric structures on
surfaces, Bull. London Math. Soc. 8, 1976, 44--48. 


\bibitem{jacobson} N. Jacobson, Lie algebras, Wiley Interscience
Publishers, New York, 1962. 

\bibitem{jacobson2} N. Jacobson, General representation theory of Jordan algebras, Trans. Amer. Math. Soc. 70, 1951, 
509--530. 

\bibitem{kiechle} H. Kiechle, Theory of K-Loops, Lect. Notes in
Math. 1778, Springer-Verlag, Berlin, 2002. 

\bibitem{structure} W. G. Lister, A structure theory of Lie triple
systems, Trans. Amer. Math. Soc. 72, 1952, 217--242. 

\bibitem{symmetric} O. Loos, Symmetric Spaces, Vol I and Vol II, Benjamin, New York, 1969, 1970. 


\bibitem{quasigroups} P. O. Miheev, L. V. Sabinin, Quasigroups
and Differential Geometry, Chapter XII, Quasigroups and Loops: Theory and
Applications, Sigma Series in Pure Math. 8, Heldermann-Verlag, Berlin, 1990, 357--430. 

\bibitem{morozov} V. V. Morozov, Klassifikation der nilpotenten Lieschen Algebren sechster Ordnung, Izv. Vyssh. Uchebn. Zaved., Mat. 4{\rm (5)}, 1958, 161--171, Russian. 

\bibitem{mubarak1} G. M. Mubarakzjanov, \"Uber aufl\"osbare Lie-algebren, Izv. Vyssh. Uchebn. Zaved., Mat. 1{\rm(32)}, 1963, 114--123, Russian. 


\bibitem{mubarak} G. M. Mubarakzjanov, Klassifikation reeller Strukturen von Lie-Algebren f\"unfter Ordnung, Izv. Vyssh. Uchebn. Zaved., Mat. 3{\rm (34)}, 1963, 99--106, Russian. 

\bibitem{loops} P. T. Nagy, K. Strambach, Loops in group theory and Lie theory, de Gruyter Expositions in Math. 35,  Walter de Gruyter, Co. Berlin, New York, 2002. 

\bibitem{reinsch} M. W. Reinsch, A simple expression for the terms in the Baker -- Campbell -- Hausdorff series, J. Math. Phys. 41, 2000, 2434--2442. 


\end{thebibliography}
\end{document}